\documentclass[12pt, twoside, leqno]{article}
\usepackage{blindtext}

\usepackage{hyperref}

\hypersetup{colorlinks,linkcolor={blue},citecolor={blue},urlcolor={red}}  
\usepackage{cite}

\usepackage{amsmath,amsthm}
\usepackage{amssymb}
\newcommand{\eps}{\varepsilon}

\usepackage{enumitem}

\usepackage{graphicx}

\newtheorem{theorem}{Theorem}[section]

\newtheorem{lemma}[theorem]{Lemma}
\newtheorem{proposition}[theorem]{Proposition}

\theoremstyle{definition}

\newtheorem{remark}[theorem]{Remark}

\numberwithin{equation}{section}

\newcommand{\Rr}{\mathbb R}

\newcommand{\Bb}{\mathbb B}
\begin{document}


\baselineskip=17pt


\title{An Asymptotic Formula for Eigenvalues of the Neumann Laplacian in Domains with a Small Star-shaped Hole}

\author{Ly Hong Hai\\
Department of Mathematics\\ 
University of Ostrava\\
30. dubna 22\\
701 03  Ostrava, Czech Republic\\
E-mail: hai.ly.s01@osu.cz}

\date{}
\maketitle


\renewcommand{\thefootnote}{}


\renewcommand{\thefootnote}{\arabic{footnote}}
\setcounter{footnote}{0}
\begin{abstract}
This article investigates a spectral problem of the Laplace operator in a two-dimensional bounded domain perforated by a small arbitrary star-shaped hole and on the smooth boundary of which the Neumann boundary condition is imposed. It is proved that the eigenvalues of this problem converge to the eigenvalues of the Laplacian defined on the unperturbed domain as the size of the hole approaches zero. Furthermore, our main theorem provides the rate of convergence by showing
an asymptotic expansion for all simple eigenvalues with respect to the size and shape of the hole.
\end{abstract}

\numberwithin{equation}{section}
\section{Introduction}
In recent years, the application of spectral theory to the study of the effects of a small hole on the sound of a drum or bell, as well as the behaviour of the heat equation in the study of heat distribution in objects, has arisen great interest among scientific
community. A large number of studies have been conducted in this field, exploring topics such as the spectral properties of Laplacian in bounded domains under Dirichlet or Neumann conditions, the asymptotic expansion of eigenvalues for domains with small holes, the changes in eigenvalues as the domain is varied and so on. Several studies are briefly discussed as follows. Rauch and Taylor \cite{rauch1975potential} found that the spectrum of Laplacian in a bounded domain remains unchanged after applying Dirichlet conditions to a compact subset of capacity zero, while Ozawa \cite{osawa1992nonlinear,ozawa1981singular,ozawa1982asymptotic,ozawa1983spectra,ozawa1984random,ozawa1985asymptotic,ozawa1986fluctuation,ozawa1992singularI,ozawa1992singularII,fozawa1980singular}, Besson \cite{besson1985comportement} and Courtois \cite{courtois1995spectrum} investigated the expansion of eigenvalues in cases of small holes with Dirichlet or Neumann boundary conditions. Kato \cite{kato2013perturbation}, Hubert \cite{hubert2012vibration}, as well as Ward and Keller \cite{ward1993strong} focused on how small changes in the domain affect the eigenvalues. Gadyl’shin and Il’in \cite{gadyl1998asymptotic}, Daners \cite{daners2003dirichlet}, and McGillivray \cite{mcgillivray1998capacitary} were considered as notable studies in this area. In \cite{maz2000asymptotic}, Maz'ya, Nazarov, and Plamenevskij investigated the Laplace operator on domains with obstacles and found a complete asymptotic expansion for eigenvalues. In contrast, Barseghyan et al. \cite{barseghyan2022neumann} demonstrated the convergence of the Laplacian spectrum on a hole-like compact set with zero Lebesgue measure to the original domain's Laplacian spectrum.

As far as we are concerned, much of the previous research has focused on studying the spectral problem of the Laplace operator in domains with specifically shaped holes, such as spherical or elliptical holes. However, the exploration of domains with holes of general shape and smooth boundaries is still limited, although they offer practical advantages. General-shaped holes with smooth boundaries are commonly encountered in engineering scenarios, including the design of heat exchangers, acoustic chambers, fluid flow control devices, and electromagnetic waveguides. By investigating the spectral problem in these domains, we can analyze the influence of the hole's shape on various physical phenomena. The smoothness of the hole's boundary ensures the continuity of heat flow, acoustic waves, fluid motion, or electromagnetic fields, allowing for accurate modelling and efficient analysis.

Motivated by the observation above, we consider the spectral problem of the Laplace operator in a bounded planar domain $\Omega$. This domain is perforated by a small star-shaped hole $E_\varepsilon$ with a smooth boundary, where $\varepsilon$ is a fixed parameter. The notion of a star-shaped set is known as a natural generalization of a convex set. We impose the Neumann condition on the boundary of the hole and the Dirichlet condition on the outer boundary of the domain. Our main contribution in this paper is constructing an asymptotic expansion with respect to the hole size for all simple eigenvalues. To analyze the problem, we employ asymptotic techniques and perturbation methods. More precisely, we first parameterize the small hole by using curvilinear coordinates, which helps to develop approximations for eigenvalues and eigenfunctions by employing the concept of Green's kernel. Green's kernel has been well-known as a fundamental solution of the Laplace equation that characterizes the behaviour of the Laplace operator in the domain of interest. We construct Green operators by utilizing Green's kernel, an integral operator associated with Laplacian. Using Green operators, we derive approximations for eigenvalues and eigenfunctions of the Laplacian with the Neumann boundary condition on the hole. These approximations provide a means to estimate the eigenvalues and eigenfunctions. From that, we obtain the desired asymptotic formula of eigenvalues. The case $N=3$, the singularity of Green function $K(x,y)$ near $x=y$ is stronger than that of the case $N=2$. When we use the Sobolev embedding; $W^{2,p}(\Omega)\hookrightarrow C^{2-N/p}(\bar\Omega)$, we must take $p$ larger as $N$ increases.

 The paper is organized as follows. In the next section, we state the problem and present our main result, which is contained in Theorem \ref{theo2.1}. Section 3 is devoted to preliminary results, in which we establish useful lemmas for estimating the $L^p$ norm used for the rest of the paper. In Section 4, using properties of Green functions and operators related to our problem, we address technical issues relevant to our main result, whose proof appears in Section 5. The last section contains conclusions and discussions.
 \section{ Statement of the problem and main result}
 Let $\Omega \subset \Rr^2$ be a bounded domain with smooth $C^1$ boundary $\partial \Omega$ and the origin $\tilde{w} \in \Omega$. Let $E\subset \Omega$ be an open domain of star-shaped concerning a neighbourhood of $\tilde{w}$ and $\partial E$ is also of class $C^1$. For $\eps>0$, we consider the small hole $E_\eps = \eps E$. We introduce the curvilinear coordinates, where $x = (x_1, x_2)$ denotes the Cartesian coordinates in $\Rr^2$ as
 \begin{align*}
     \begin{cases} 
	    x_1 = \alpha(\rho) \beta(\theta) \cos \theta \\
        x_2 = \alpha(\rho) \omega(\theta)\sin\theta
	\end{cases}
 \end{align*}
	with $\rho \in [0,\infty) $ and $\theta \in \Rr$, where the functions $\alpha, \beta, \omega$ are $C^1$-functions satisfying
 \begin{itemize}
     \item[(i)] $\beta, \omega$ are periodic of period $2\pi$ with $\beta(0), \beta(\pi), \omega(\pi/2), \omega(3\pi/2) \ne 0$,
     \item[(ii)] $\alpha, \beta$ are strictly increasing and $\alpha(0) = 0$, $\lim_{\rho\to\infty}\alpha(\rho) = \infty$ and
     \item[(iii)] for any $\rho>0, \theta\in [0,2\pi)$
     \begin{align*}
         \beta(\theta)\cos(\theta) (\beta(\theta)\cos(\theta))' + \omega(\theta)\sin(\theta)(\omega(\theta)\sin(\theta))'  =0.
     \end{align*}
 \end{itemize}
 The assumptions (i) and (ii) ensure that map $(\rho,\theta) \mapsto (x_1,x_2)$ is a diffeomorphism from $(0,\infty)\times [0,2\pi)$ to $\Rr\backslash \{\tilde{w}\}$. The condition (iii) is for the orthogonality of the curve $(x_1(\cdot,\theta_0),x_2(\cdot,\theta_0))$ and $(x_1(\rho_0,\cdot),x_2(\rho_0,\cdot))$ for any $\rho_0>0$ and $\theta_0\in [0,\infty)$, which is important for the sequel analysis. In the special case, it is easy to see that choice $\alpha(\rho) = e^\rho-1$, $\beta(\theta) = 2+ \cos(4\theta)$, and $\omega(\theta)=2+ \frac{1}{6} \frac{\sin(3\theta)}{\sin{\theta}}+\frac{1}{10}\frac{\sin(5\theta)}{\sin\theta}$ satisfies these three assumptions. Moreover, thanks to the assumption of $E$, we can find $\rho_0>0$ and functions $\alpha,\beta,\omega$ satisfying (i)--(iii) and the boundary $\partial E$ is characterized by
 \begin{equation*}
     \partial E = \{(x_1(\rho_0,\theta),x_2(\rho_0,\theta)): \theta\in [0,2\pi)\}.
 \end{equation*}
 Therefore, the small hole $E_\eps$ and its boundary is characterized by
 \begin{equation*}
     E_\eps = \{(\eps  x_1(\rho,\theta),\eps x_2(\rho,\theta)): \; 0\le \rho < \rho_0, \theta\in [0,2\pi)\}
 \end{equation*}
 and
 \begin{equation*}
     \partial E_\eps = \{(\eps x_1(\rho_0,\theta),\eps x_2(\rho_0,\theta)): \theta\in [0,2\pi)\}.
 \end{equation*} 
Define for $x \in \partial E_\varepsilon$ the function $f(x)=|x|^2$. Then, $f$ is continuous thanks to the smoothness of $\partial E_\varepsilon$. Note that there always exist two points $x_{\min},x_{\max}\in\partial E_\varepsilon$ such that $f(x_{\min}) = \min_{x\in\partial E_\varepsilon}f(x)$ and $f(x_{\max}) = \max_{x\in\partial E_\varepsilon}f(x)$. It is immediate that
\[\pi |x_{\min}|^2 \le |E_\varepsilon| \le \pi|x_{\max}|^2.\]
Since the continuous function $g(x)= f(x) - |E_\varepsilon|/\pi$ satisfies $g(x_{\min}) g(x_{\max}) \le 0$, there exists at least one point $x^* \in \partial E_\varepsilon$ such that $g(x^*)=0$, or equivalently, 
\begin{align}
    |E_\varepsilon|=\pi|x^*|^2 = \pi(M\varepsilon)^2 \quad \text{ where }\quad  M= \frac{|x^*|}{\varepsilon} = \left(\frac{|E|}{\pi}\right)^{1/2}. 
    \label{HaiThamGraz} 
\end{align}
\medskip
We denote $\Omega_{E_\varepsilon}:= \Omega \setminus \bar{E}_\varepsilon$ and consider the following eigenvalue problem
\begin{equation} \label{eq2.1}
    \begin{cases}
         - \Delta  u(x) = \lambda(\varepsilon)u(x) &x\in \Omega_{E_\varepsilon},\\
	u(x)= 0 &x \in \partial \Omega, \\
	\frac{\partial u}{\partial \nu}(x)  = 0  &x \in \partial E_\varepsilon.
    \end{cases}
\end{equation}
Here the symbol $\partial u/ \partial \nu$ is the directional derivative of $u$ along the unit outward normal vector $\nu$ to the domain $\Omega_{E_\varepsilon}$, hence pointing the spectral problem mentioned above has a countable sequence of strictly positive eigenvalues, with each eigenvalue having a finite multiplicity \cite{schmudgen2012unbounded}. These eigenvalues are conventionally arranged in ascending order, and if an eigenvalue has a multiplicity greater than one, it is repeated accordingly:
\[ 0 < \mu_1(\varepsilon) < \mu_2(\varepsilon) \le \dots \le \mu_i(\varepsilon) \le \dots \to \infty.\]
A crucial property is that as $\varepsilon$ approaches zero, see e.g. Rauch \& Taylor \cite{rauch1975potential}, each spectral problem's eigenvalue converges to the Dirichlet Laplacian's corresponding eigenvalue in the original domain without a hole. More precisely,
\begin{align}\label{eq2.2}
    \mu_i(\varepsilon) \to \mu_i, \text{   as } \varepsilon \to 0,
\end{align}
where  $\{\mu_i\}$ are the eigenvalues of the following problem
\begin{equation}\label{eq2.3}
\begin{cases}
    - \Delta u(x) = \lambda u(x), & x \in \Omega \\
	u(x)= 0, & x \in \partial \Omega.
\end{cases}
\end{equation}

Let $K(x,y)$ represent the Green function of the Laplacian in the domain $\Omega$, associated with the boundary condition mentioned in equation \eqref{eq2.3}. Similarly, $K_\varepsilon(x,y)$ denotes the Green function in the domain $\Omega_{E_\varepsilon}$, associated with the boundary condition specified in equation \eqref{eq2.1}.

Let $\varphi_i(x)$ denote the eigenfunction of \eqref{eq2.1} corresponding to $\mu_i$, which is normalized in $L^2(\Omega)$. To indicate the asymptotic behaviour, we employ the notation $\mathcal{O}(\cdot)$, commonly known as the ``big-$\mathcal{O}$'' Landau notation. Specifically, for any sufficiently small positive $x$, the notation $f(x)=\mathcal{O}(x)$ implies that $\lvert f(x) \rvert \le Cx$, where $C$ is a constant.

We will present our main result, which provides the convergence rate \eqref{eq2.2}.
\begin{theorem}\label{theo2.1}
Assume that $\mu_i$ is a simple eigenvalue. Then, as $\varepsilon$ tends to zero, we have 
	\[\mu_i(\varepsilon)  =\mu_i  - (2 \lvert \nabla \varphi_i(\tilde{w})\rvert^2 -\mu_i \varphi_i(\tilde{w})^2) |E|\varepsilon^2  +  \mathcal{O}(\varepsilon^3|\log \eps|^2)\]
 	where $\mu_i$ is the corresponding eigenvalue of the unperturbed problem to $\mu_i(\varepsilon)$, $\tilde{w}$ is the origin and $|E|$ is the area of $E$.
\end{theorem}
Our result extends the analysis of eigenvalue perturbations to two dimensional domains, providing a distinct asymptotic expansion. While \cite{maz2000asymptotic} Maz'ya's three-dimensional result, which shows a leading order term scaling with \(\eps^3\), our findings indicate that in two dimensions, the leading order term scales with \(\eps^2\). Additionally, we prove a higher-order term \(\mathcal{O}(\eps^3 |\log \eps|^2)\) for greater precision. This highlights the different behaviours of eigenvalue perturbations in various dimensional settings and underscores the importance of dimensionality in asymptotic analysis.

\begin{remark}\label{remark1}\hfill
\begin{itemize}
    \item The asymptotic in Theorem \ref{theo2.1} does not depend on the parameterization of $E$, but only on the area of $E$. This also implies that the asymptotic behaviour is invariant with respect to rotations and translations.
    \item Our results and proofs are not applicable to eigenvalues with multiplicity higher than one. This case is for future investigation (see e.g. \cite{abatangelo2022ramification}).
\end{itemize}
\end{remark}
\section{Preliminary lemmas}
For any periodic function $L(\theta)$  of $\theta \in [0, 2 \pi]$,  its  Fourier series is given by
\begin{align*}
    L(\theta) = S_0 + 
\sum_{k=0}^{\infty} S_k \cos k\theta + P_k \sin k \theta.
\end{align*} 
\begin{lemma} \label{lemma3.1}
Consider the equation system described by:
\begin{align}
\Delta v_\varepsilon(x) &= 0 \quad \text{for } x \in \mathbb{R}^2 \setminus \bar{E}_\varepsilon, \label{eq3.1} \\
\frac{\partial v_\varepsilon}{\partial \nu}(x) &= L(\theta) \quad \text{for } x \in \partial E_\varepsilon, \label{eq3.2}
\end{align}
where $L(\theta)$ is a given smooth function on $\partial E_\varepsilon$. Then there exists at least one solution $v_\varepsilon$ that satisfies the following condition:
\begin{equation}\label{eq3.3}
\lvert v_\varepsilon(x)\rvert \le C \varepsilon \max_{\theta \in [0,2 \pi]} \lvert L(\theta)\rvert (1+\rho+|\log \eps|),
\end{equation}
where $C$ is a constant independent of $\varepsilon$.
\end{lemma}	
\begin{proof}
	
\noindent Now, consider the function:

	\[  v_\varepsilon(x) = A_0  (\rho + \log \eps) + \sum_{k=1}^{\infty} (-k)^{-1}e^{-k\rho} (A_k \cos k \theta + B_k \sin k \theta).\]
    Thanks to the assumption (iii) of the curvilinear coordinates, we have
    \begin{equation*}
        \Delta v_\eps = \partial_{x_1}^2v_\eps + \partial_{x_2}^2v_\eps =  \left[(\partial_\rho x_1)^2 + (\partial_\rho x_2)^2\right]^{-1}(\partial_{\rho}^2 v_\eps + \partial_{\theta}^2v_\eps).
    \end{equation*}
    Therefore, direct computations lead to $\Delta v_\varepsilon(x)=0$ for $x \in \mathbb{R}^2 \setminus \bar{E}_\varepsilon$. 
	Next, we consider the normal derivatives $\frac{\partial}{\partial \nu}$ of $\partial E_\varepsilon$. We have
 \begin{align*}
    d\nu=\sqrt{\left( \frac{\partial x_1}{\partial \rho}\right)^2 + \left( \frac{\partial x_2}{\partial \rho}\right)^2} d\rho &= \varepsilon |\alpha'(\rho)|\sqrt{ \beta(\theta)^2 \cos^2 \theta +  \omega(\theta)^2 \sin^2 \theta  }\, d\rho\\
    &= \varepsilon \tau(\rho,\theta) \,d\rho
\end{align*}
where $\tau(\rho,\theta)=|\alpha'(\rho)|\sqrt{\beta(\theta)^2 \cos^2 \theta +  \omega(\theta)^2 \sin^2 \theta  } \ne 0$ due to the assumptions on $\alpha, \beta, \omega$.

Differentiating the function $v_\varepsilon(x)$ along the exterior normal vector at $x \in \partial E_\varepsilon$, we observe that
	\begin{eqnarray*}
		&&\frac{\partial v_\varepsilon}{\partial \nu}(x) \vert_{x \in \partial E_\varepsilon} \\
		& =& \left. 	 \frac{\partial v_\varepsilon}{\partial \rho}\frac{\partial \rho }{\partial \nu } \right\rvert _{ x \in \partial E_\varepsilon}\\
		&=& \left. \frac{\partial v_\varepsilon}{\partial \rho}\frac{1}{\varepsilon \tau(\rho,\theta)}  \right \rvert_{\rho=\rho_0} \\
		&=&  \frac{1}{\varepsilon  \tau(\rho_0,\theta) } \left(  A_0 + \sum_{k=1}^{\infty}e^{-k\rho_0}(A_k \cos k\theta+ B_k \sin k \theta)\right)   \\
		&=&S_0 + \sum_{k=1}^{\infty}(S_k \sin k\theta + P_k \cos k\theta)  \\
		&=& L(\theta),
	\end{eqnarray*}
	which implies 
	\begin{eqnarray*}
		A_0 & =& \varepsilon   \tau(\rho_0,\theta) S_0 , \\
		A_k &=& \varepsilon e^{-k\rho_0}  \tau(\rho_0,\theta) S_k  , \\
		B_k &=&  \varepsilon  e^{-k\rho_0}  \tau(\rho_0,\theta) P_k , 
	\end{eqnarray*}
	for $k \ge 1$. 
Thus, we rewrite the inequality as
	\begin{multline}\label{eq3.4}
	     \lvert v_\varepsilon(x)\rvert \le \lvert S_0\rvert  \varepsilon \tau(\rho_0,\theta) (\rho+ |\log \eps|) + \left(  \sum_{k=1}^{\infty}  \left(S_k^2 + P_k^2\right)\right) ^{1/2}\\
      \times \left( \sum_{k=1}^{\infty}(-k)^{-2}\varepsilon ^2 \tau(\rho_0,\theta)^2 e^ {-2k (\rho_0+\rho)}\right) ^{1/2}.
	\end{multline}
By applying Schwarz's inequality, we have
	\begin{align}\label{eq3.5}
		 \sum_{k=1}^{\infty}(-k)^{-2}\varepsilon^2 \tau(\rho_0,\theta)^2 e^{-2k(\rho_0+\rho)}
		&\le \varepsilon^2 \tau(\rho_0,\theta)^2 \left(    \sum_{k=1}^{\infty} \frac{1}{k^4}\right)^{1/2}  \left(    \sum_{k=1}^{\infty} e^{-4k(\rho_0+\rho)}\right)^{1/2} \nonumber\\
		&\le C \tau(\rho_0,\theta)^2\varepsilon^2. 
	\end{align}
	Combining equations (\ref{eq3.4}) and (\ref{eq3.5}), along with the inequality
	\[ S_0^2+ \sum_{k=1}^{\infty} \left(S_k^2 +  P_k^2\right) \le C \int_{0}^{2\pi} \lvert L(\theta)\rvert^2 d\theta \le  C'\max_{\theta \in[0,2\pi]} \lvert L(\theta)\rvert^2,\]
    we obtain the estimate
	\begin{eqnarray*}
 \lvert v_\varepsilon(x)\rvert &\le& 
  C  \varepsilon\max_{\theta \in [0,2\pi]} \lvert L(\theta)\rvert (1+\rho+|\log \eps|).
	\end{eqnarray*}	
Hence, the proof is complete.
\end{proof}
\begin{lemma} \label{lemma3.2}
    Let $N \in C^\infty(\partial E_\varepsilon)$ and $g \in H^2(\Omega_{E_\varepsilon})$ satisfy the following conditions:
    \begin{equation} \label{eq3.6}
    \begin{cases}
        \Delta g(x) =0 & x\in \Omega_{E_\varepsilon}, \cr
        g(x) = 0 & x \in \partial\Omega, \cr
        \frac{\partial g}{\partial \nu}(x) = N(x) &x \in \partial E_\varepsilon.
    \end{cases}
    \end{equation}
    Then, the following inequality holds:
    \[ \int_{\Omega_{E_\varepsilon}} \lvert \nabla g(x)\rvert^2 dx \le  2 \pi\, \varepsilon^2 M_\textup{max} \left(\max_{\partial E_\varepsilon} \lvert N(x)\rvert  \right)^2\]
    where $M_{\max} = \sup_{x\in\partial E}|x|$.
\end{lemma}
\begin{proof} 
 Since $g\in H^2(\Omega_{E_\varepsilon})$, applying Green’s formula on $\Omega_{E_\varepsilon}$, we get
\[ \int_{\Omega_{E_\varepsilon}} (g \Delta g+ \lvert \nabla g\rvert^2)dx = \int_{\partial \Omega_{E_\varepsilon}} g \frac{\partial g}{\partial \nu}d\sigma_x.\]
By (\ref{eq3.6}), it follows that 
\begin{eqnarray} \label{eq3.7}
\int_{\Omega_{E_\varepsilon}} \lvert\nabla g\rvert^2dx = \int_{\partial E_\varepsilon} g \frac{\partial g}{\partial \nu} d\sigma_x.
\end{eqnarray}
Using Schwarz's inequality, we have
\begin{eqnarray} \label{eq3.8}
\int_{\partial E_\varepsilon} g(x) \frac{\partial g}{\partial \nu}(x) d\sigma_x 
    &\le& \left( \int_{\partial E_\varepsilon} g(x)^2 d\sigma_x \right)^{1/2} \left( \int_{\partial E_\varepsilon} \left\lvert\frac{\partial g}{\partial \nu}(x) \right\rvert^2 d\sigma_x\right)^{1/2} \nonumber\\
    &\le&  (2 \pi\,  \eps M_{\max})\eps|\log \eps|  \max_{\partial E_\varepsilon}\lvert N(x)\rvert \,\max_{\partial E_\varepsilon} \lvert N(x)\rvert \nonumber\\
     &\le&  2 \pi\,M_\text{max} \varepsilon^2 |\log \eps| \left(\max_{\partial E_\varepsilon} \lvert N(x)\rvert  \right)^2.
\end{eqnarray}
    The last inequality follows from the maximum principle, $g$ is harmonic on $\mathbb{R}^2\setminus{\bar E_\varepsilon}$, which satisfies the assumption of Lemma \ref{lemma3.1}. Then, we have $\lvert g(x) \rvert \leq C \eps |\log \eps| \max_{\partial E_\varepsilon} \lvert N(x) \rvert$ for all $x \in \Omega_{E_\varepsilon}$.

From \eqref{eq3.7} and \eqref{eq3.8}, we deduce
 \[ \int_{\Omega_{E_\varepsilon}} \lvert \nabla g(x)\rvert^2 dx \le   2 \pi\, \eps^2|\log \eps| M_\text{max} \left(\max_{\partial E_\varepsilon} \lvert N(x)\rvert  \right)^2.\]
\end{proof}
\begin{lemma}\label{lemma3.3}
	Fix $M \in C^\infty ( \bar{\Omega}_{E_\varepsilon})$. Let $u$ be the solution of
	\begin{equation}\label{eq3.9}
            \begin{cases}
                 \Delta u_\varepsilon(x) =0 & x \in \Omega_{ E_\varepsilon},\\
	u_\varepsilon(x) = 0 & x \in \partial \Omega, \\
	\frac{\partial}{\partial \nu }u_\varepsilon(x) = M(\theta) &  x \in \partial E_\varepsilon.
            \end{cases}
	\end{equation} 
Then, $u_\varepsilon(x)$ satisfies 
	\[ \lvert u_\varepsilon(x)\rvert \le  C\varepsilon |\log \eps| \max_{\theta \in [0,2\pi]} \lvert M(\theta)\rvert \]
	and, consequently,
	\[ \|u_\varepsilon\|_{L^2(\Omega_{E_\varepsilon})} \le C \varepsilon |\log \eps| \max_{\theta \in [0,2\pi]} \lvert M(\theta)\rvert.\]
\end{lemma}
\begin{proof}
Let $u_\varepsilon(x)$ be a solution to the problem \eqref{eq3.9}. Initially, we choose $L(\theta) = M(\theta)$ and obtain $v_\varepsilon^{(0)}(x)$ as a solution that satisfies equations \eqref{eq3.1}, \eqref{eq3.2}, and \eqref{eq3.3}. However, $v_\varepsilon^{(0)}(x)$ may not satisfy $v_\varepsilon^{(0)}(x) = 0$ for $x \in \partial \Omega$. To address this, let $v_\varepsilon^{(1)}(x)$ be the harmonic function in $\Omega$ that satisfies $v_\varepsilon^{(1)}(x) = v_\varepsilon^{(0)}(x)$ for $x \in \partial \Omega$. We define
$$ M_\varepsilon= \max_{\theta \in [0,2\pi]} \lvert M(\theta)\rvert. $$
 From \eqref{eq3.3} we observe that $$\max \{\lvert v_\varepsilon^{(1)}(x)\rvert, x \in \bar{\Omega}\} \le \hat{C} \varepsilon |\log \eps|M_\varepsilon$$ and $\max\{\lvert \frac{\partial }{\partial \nu}v_\varepsilon^{(1)}(x) \rvert, x \in \partial E_\varepsilon\}\le \hat{C} \varepsilon |\log \eps| M_\varepsilon$, where $\hat{C}$ is a constant independent of $\varepsilon$. 
 Next, we set $L(\theta)= \frac{\partial }{\partial \nu}v_\varepsilon^{(1)}(x)$ for $ x \in \partial E_\varepsilon $ and consider $v_\varepsilon^{(2)}(x)$ as a solution satisfying equations \eqref{eq3.1}, \eqref{eq3.2}  and \eqref{eq3.3}. Let $v_\varepsilon^{(3)}(x)$ be the harmonic function in $\Omega$ satisfying $v_\varepsilon^{(3)}(x)= v_\varepsilon^{(2)}(x)$ for $x \in \partial \Omega$. Then, $\max \{\lvert v_\varepsilon^{(3)}(x)\rvert, x \in \bar{\Omega}\} \le  (\hat{C}\varepsilon |\log \eps|)^2 M_\varepsilon$ and $\max \{\lvert\frac{\partial}{\partial \nu}v_\varepsilon^{(3)}(x)\rvert, x \in \partial E_\varepsilon\}\le  (\hat{C}\varepsilon |\log \eps|)^{2}M_\varepsilon$.
 
\noindent By repeating this procedure, we obtain the following equations for each $n=0,1,2,\ldots$

\noindent For $v_\varepsilon^{(2n+1)}(x)$:
\begin{eqnarray*}
\Delta v_\varepsilon^{(2n+1)}(x) &=& 0 \quad \text{if } x \in \Omega, \\
v_\varepsilon^{(2n+1)}(x) &=& v_\varepsilon^{(2n)}(x) \quad \text{if } x \in \partial \Omega.
\end{eqnarray*}

\noindent For $v_\varepsilon^{(2n+2)}(x)$:
\begin{eqnarray*}
\Delta v_\varepsilon^{(2n+2)}(x) &=& 0 \quad \text{if } x \in \mathbb{R}^2 \backslash \bar{E}_\varepsilon, \\
\frac{\partial}{\partial \nu}v_\varepsilon^{(2n+2)}(x) &=& \frac{\partial}{\partial \nu}v_\varepsilon^{(2n+1)}(x) \quad \text{if } x \in \partial E_\varepsilon.
\end{eqnarray*}

\noindent By induction, we have the following inequalities:
\begin{eqnarray}
\max_{\bar{\Omega}} \lvert v_\varepsilon^{(2n+1)}(x)\rvert &\le& (\hat{C} \varepsilon |\log \eps|)^{n+1}M_\varepsilon, \label{eq3.10} \\
\max_{\partial E_\varepsilon} \left\lvert \frac{\partial}{\partial \nu}v_\varepsilon^{(2n+1)}(x) \right\rvert &\le& (\hat{C} \varepsilon |\log \eps|)^{n+1}M_\varepsilon, \label{eq3.11} \\
\lvert v_\varepsilon^{(2n)}(x) \rvert &\le& (\hat{C}\varepsilon |\log \eps|)^{n+1} M_\varepsilon \label{eq3.12}
\end{eqnarray}
     hold for $n \ge 0$. 
     We can choose $\varepsilon$ such that $\hat{C}\varepsilon|\log \eps| \le 1/2$. Let us define
\begin{eqnarray}
w_\varepsilon(x) = \sum_{n=0}^\infty (-1)^n v_\varepsilon^{(n)}(x). \label{eq3.13}
\end{eqnarray}
From \eqref{eq3.10} and \eqref{eq3.12}, we observe that the right-hand side of \eqref{eq3.13} uniformly converges on $\bar{\Omega}\backslash E_\eta$ for any $\eta >\varepsilon$. Since $v_\varepsilon^{(n)}$ is harmonic in $\Omega_{E_\varepsilon}$, it follows that $w_\varepsilon(x)$ is harmonic in $\Omega_{E_\varepsilon}$, $w_\varepsilon(x) = 0$ for $x \in \partial \Omega$, and
\begin{eqnarray*}
\frac{\partial w_\varepsilon}{\partial x_j}(x) = \sum_{n=0}^\infty (-1)^n \frac{\partial v_\varepsilon^{(n)}}{\partial x_j}(x) \quad \text{for } x \in \Omega_{E_\varepsilon}, \quad j = 1, 2.
\end{eqnarray*}
Next, we define
\begin{eqnarray*}
g_\varepsilon^{(n)}(x) = u_\varepsilon(x) - \sum_{i=0}^{2n+1}(-1)^i v_\varepsilon^{(i)}(x).
\end{eqnarray*}
Then, 
\begin{eqnarray} \label{eq3.14}
         \nabla g_\varepsilon^{(n)}(x) \to \nabla (u_\varepsilon - w_\varepsilon)(x) \qquad (n\to \infty) \text{  a.e. } x \in \Omega_{E_\varepsilon}.
     \end{eqnarray}
We have the properties that $g_\varepsilon^{(n)}(x)$ is harmonic in $\Omega_{E_\varepsilon}$, $g_\varepsilon^{(n)}(x) = 0$ for $x \in \partial \Omega$, and
\begin{eqnarray*}
\frac{\partial g_\varepsilon^{(n)}}{\partial \nu}(x) = \frac{\partial v_\varepsilon^{(2n+1)}}{\partial \nu}(x) \quad \text{for } x \in \partial E_\varepsilon.
\end{eqnarray*}
Therefore, applying Lemma \ref{lemma3.2} and \eqref{eq3.11}, we obtain
\begin{eqnarray*}
\int_{\Omega_{E_\varepsilon}} \lvert \nabla g_\varepsilon^{(n)}\rvert^2 dx &\le& 2\pi M_\text{max} \varepsilon^2 |\log \eps|  \left( \max_{\partial E_\varepsilon} \left\lvert \frac{\partial v_\varepsilon^{(2n+1)}}{\partial \nu}(x) \right\rvert \right)^2 \\
&\le& 2\pi M_\text{max} \varepsilon^2|\log \eps| (\hat{C}\varepsilon|\log \eps|)^{2n+2}M_\varepsilon^2.
\end{eqnarray*}
By using Fatou's Lemma and \eqref{eq3.14}, we deduce that
     \[ \int_{\Omega_{E_\varepsilon}} \lvert\nabla(u_\varepsilon - w_\varepsilon)\rvert^2 dx \le \liminf_{n \to \infty} \int_{\Omega_{E_\varepsilon}} \lvert\nabla g^{(n)}\rvert^2 dx \le 0.\]
     Thus, $u_\varepsilon - w_\varepsilon =$ constant a.e $\Omega_{E_\varepsilon}$. Since $u_\varepsilon(x) = w_\varepsilon(x)=0$ for $x \in \partial \Omega$, it follows that $u_\varepsilon = w_\varepsilon$ a.e. $\Omega_{E_\varepsilon}$. Therefore,
     \begin{eqnarray} \label{eq3.15}
         u_\varepsilon(x) = \sum_{n=0}^{\infty} (-1)^n v_\varepsilon^{(n)} \qquad x\in \Omega_\varepsilon.
     \end{eqnarray}
    From \eqref{eq3.10} and \eqref{eq3.12}, we have
    \begin{eqnarray*}
    \lvert u_\varepsilon(x)\rvert &\le&\sum_{n=0}^{\infty}(\lvert v_\varepsilon^{(2n)}(x)\rvert + \lvert v_\varepsilon^{(2n+1)}(x)\rvert ) \\
                       &\le & C \varepsilon|\log \eps| M_\varepsilon.  
    \end{eqnarray*} 
Consequently, we can conclude that the function $u_\varepsilon$ satisfies the following inequality:
    	\[ \|u_\varepsilon\|_{L^2(\Omega_{E_\varepsilon})} \le C \varepsilon |\log \eps|\max_{\theta \in [0,2\pi]} \lvert M(\theta)\rvert,\]
     where $C$ is constant independent of $\varepsilon$.
\end{proof}
 
\section{Some technical results}
\subsection{Green operators in perforated domains and their approximation}
Consider the Green function $K(x,y)$ of the Laplacian in $\Omega$ under the Dirichlet condition on $\partial \Omega$. Similarly, let $K_\varepsilon(x,y)$ be the Green function of the Laplacian in $\Omega_{E_\varepsilon}$ satisfying
\begin{align*}
    \begin{cases}
	- \Delta_x K_\varepsilon (x,y)  = \delta(x-y), & x,y \in \Omega_{E_\varepsilon} \\
	K_\varepsilon(x,y)\vert_{x\in \partial \Omega}  =0, & y \in \Omega_{E_\varepsilon}\\	
	\frac{\partial}{\partial \nu}  K_\varepsilon (x,y)\vert_{x\in \partial E_\varepsilon} = 0,  & y \in  \Omega_{E_\varepsilon}. 
\end{cases}
\end{align*}

We define the bounded linear operators $\mathbf{K}$ and $\mathbf{K}_\varepsilon$ on $L^2(\Omega)$ and $L^2(\Omega_{E_\varepsilon})$, respectively, as follows:
\begin{eqnarray*}
	(\mathbf{K}f)(x) &=& \int_\Omega K(x,y)f(y)dy,\\
	(\mathbf{K_\varepsilon}f)(x)&=& \int_{\Omega_{E_\varepsilon}}K_\varepsilon(x,y)f(y)dy,
\end{eqnarray*}
By utilizing these operators, Problems \eqref{eq2.1} and \eqref{eq2.3} can be reformulated as
\begin{eqnarray*}
	(\mathbf{K}_\varepsilon u)(x) &=& \lambda(\varepsilon)^{-1}u(x),\\
	(\mathbf{K}v)(x) &=& \lambda^{-1}v(x).
\end{eqnarray*}
In order to establish a connection between $\mathbf{K_\varepsilon}$ and $\mathbf{K}$, we introduce the operators $\mathbf{H_\varepsilon}$ and $\mathbf{\overline{H}_\varepsilon}$.  First, we define the inner product as follows:
\[\left\langle   \nabla_w a(x,w), \nabla_w b(w,y)  \right\rangle = \sum_{i=1}^{2} \frac{\partial}{\partial w_i} a(x,w)  \frac{\partial}{\partial w_i } b(w,y)\]
for any $a,b \in C^1(\Omega \times \Omega \backslash (\Omega \times \Omega)_d)$, where $(\Omega \times \Omega)_d$ represents the diagonal set of $\Omega \times \Omega$. The inner product $\left\langle \nabla_w, \nabla_w \right\rangle$ remains invariant under any orthogonal transformation of orthonormal coordinates $(w_1, w_2)$.

Similarly, by considering $a(x,w)$ and $b(x,w)$ belonging to $C^3(\Omega\backslash \{w\})$ and satisfying $\Delta_w a(x,w) = \Delta_w b(x,w) = 0$ for $x,y \in \Omega\backslash\{w\}$, we define
\[\left\langle   H_w a(x,w), H_w b(w,y)  \right\rangle = \sum_{i,j=1}^{2} \frac{\partial^2}{\partial w_i \partial w_j} a(x,w)  \frac{\partial^2}{\partial w_i \partial w_j } b(w,y).\]
It is noteworthy that $ \left\langle  H_w a(x,w), H_w b(w,y)\right\rangle\vert_{x,y \in \Omega \backslash \{w\}}$ remains invariant under any orthonormal transformation of the basis.

With $M$ defined by \eqref{HaiThamGraz}, we introduce the functions $h_\varepsilon(x,y)$ and $\bar{h}_\varepsilon(x,y)$ as follows:
\begin{eqnarray*}
 h_\varepsilon(x,y) &=& K(x,y) + h(\varepsilon)\left\langle \nabla_w  K(x,\tilde{w}), \nabla_w K(\tilde{w},y)\right\rangle\\
 &&+ i(\varepsilon)\left\langle H_w  K(x,\tilde{w}), H_w K(\tilde{w},y)\right\rangle,
 \end{eqnarray*}
and
\begin{eqnarray*}
\bar{h}_\varepsilon(x,y) &=& K(x,y)+ g(\varepsilon)K(x,w) K(w,y)\\
&&+ h(\varepsilon)\left\langle \nabla_w  K(x,\tilde{w}), \nabla_w K(\tilde{w},y)\right\rangle \xi_\varepsilon(x)\xi_\varepsilon(y) \\
&&+  i(\varepsilon)\left\langle H_w  K(x,\tilde{w}), H_w K(\tilde{w},y)\right\rangle \xi_\varepsilon(x)\xi_\varepsilon(y),
\end{eqnarray*}
 where $g(\varepsilon)= -\pi \mu_i (M\varepsilon)^2$, $h(\eps)=2\pi(M\eps)^2$, $i(\eps)=\frac{\pi}{2}(M\eps)^4$, and $\xi_\varepsilon(x) \in C^\infty(\mathbb{R}^2)$ satisfies $0 \le \xi_\varepsilon(x) \le 1$, $\xi_\varepsilon(x) =1$ for $x \in \mathbb{R}^2 \setminus E_\varepsilon$ and $\xi_\varepsilon(x) = 0$ for $x \in E_{\varepsilon/2}$.

The operators $\mathbf{H}_\varepsilon$ and $\mathbf{\overline{H}}_\varepsilon$ are defined as follows
\begin{align*}
    \begin{array}{llll}
    (\mathbf{H}_\varepsilon g) (x) &=& \displaystyle \int_ {\Omega_{E_\varepsilon}} h_\varepsilon (x,y) g(y) dy, & x\in \Omega_{E_\varepsilon}, \vspace*{0.15cm} \\
    (\mathbf{\overline{H}}_\varepsilon f) (x) &=& \displaystyle  \int _{\Omega} \bar{h}_\varepsilon (x,y) f(y) dy, & x\in \Omega.
    \end{array}
\end{align*}
The operator $\mathbf{H}_\varepsilon$ serves as a very good approximation to $\mathbf{K}_\varepsilon$. Comparing $\mathbf{H}_\varepsilon$ with $\mathbf{\overline{H}}_\varepsilon$ is straightforward based on their respective definitions. Due to the fact that $\mathbf{\overline{H}}_\varepsilon$ acts on $L^2(\Omega)$, we can directly compare it with $\mathbf{K}$. Consequently, we can compare $\mathbf{K}_\varepsilon$ and $\mathbf{K}$.

To justify why $h_\varepsilon(x,y)$ provides a good approximation of $K_\varepsilon(x,y)$, let us consider the difference function $q_\varepsilon(x,y) = h_\varepsilon(x,y)- K_\varepsilon(x,y)$. It satisfies the following conditions:
\begin{eqnarray*}
\Delta_x q_\varepsilon(x,y) &=& 0, \qquad x,y \in \Omega_{E_\varepsilon},\\
q_\varepsilon(x,y)&=& 0, \qquad x \in \partial \Omega, y \in \Omega_{E_\varepsilon}.
\end{eqnarray*}
We introduce the function
\[  S(x,y) = K(x,y)+ (2\pi)^{-1}\log \lvert \nu-y \rvert,\]
which belongs to  $C^\infty( \Omega \times \Omega)$.
Take an arbitrary $x \in \partial E_\eps$, and $\nu$ is the associated normal vector of $x$. Calculating the normal derivative directly, we have
\begin{equation}\label{eq4.5}
	\begin{aligned} 
		&\frac{\partial}{\partial \nu}q_\varepsilon  (x,y)\\
  &=\frac{\partial}{\partial \nu}\Bigg( K(x,y)-K_\eps(x,y)+h(\varepsilon)\langle  \nabla_w K(x,\tilde{w}), \nabla_w K(\tilde{w},y)\rangle +i(\varepsilon)\langle H_w K(x, \tilde{w}), H_w K(\tilde{w},y)\rangle\Bigg)\\
		&= \frac{\partial}{\partial \nu} K(x,y) \\
  &\quad-h(\eps)(2\pi)^{-1} \left\langle  \frac{\partial}{\partial \nu} \nabla_w \log|\nu-\tilde{w}|, \nabla_w K(\tilde{w},y)\right\rangle \\
  &\quad+h(\varepsilon) \frac{\partial}{\partial \nu} 
  \langle  \nabla_w S(x,\tilde{w}), \nabla_w K(\tilde{w},y)\rangle\\
	&\quad -i(\eps)(2\pi)^{-1} \left\langle \frac{\partial}{\partial \nu} H_w \log|\nu-\tilde{w}|, H_w K(\tilde{w},y)\right\rangle \\
 &\quad+i(\varepsilon) \frac{\partial}{\partial \nu} 
  \langle  H_w S(x,\tilde{w}), H_w K(\tilde{w},y)\rangle.
  \end{aligned}
  \end{equation}
  Focusing on each term separately, we start with the derivatives of the logarithmic terms and their corresponding expressions
\begin{equation}\label{exterior_derivative_nabla}
    \begin{aligned}
         &\left\langle  \frac{\partial}{\partial \nu} \nabla_w \log|\nu-\tilde{w}|, \nabla_w K(\tilde{w},y)\right\rangle\\
         &=  \left\langle  \frac{\partial}{\partial \nu} \left(\frac{-\nu_1}{|\nu|^2}, \frac{-\nu_2}{|\nu|^2}\right), \nabla_w K(\tilde{w},y)\right\rangle \\
         &=-\left\langle  \left(\nu \left(\frac{\nu_2^2-\nu_1^2}{|\nu|^4}, \frac{-2\nu_1\nu_2}{|\nu|^4}\right),\nu \left( \frac{-2\nu_1\nu_2}{|\nu|^4},\frac{\nu_2^2-\nu_1^2}{|\nu|^4}\right) \right), \nabla_w K(\tilde{w},y)\right\rangle\\
         &= |\nu|^{-2}\left\langle  \nu, \nabla_w K(\tilde{w},y)\right\rangle.
    \end{aligned}
\end{equation}
\begin{equation}\label{exterior_derivative_hessian}
    \begin{aligned}
         &\left\langle \frac{\partial}{\partial \nu} H_w \log|\nu-\tilde{w}|, H_w K(\tilde{w},y)\right\rangle\\
         &=  \left\langle \frac{\partial}{\partial \nu}  
\begin{pmatrix}
  \frac{\nu_2^2-\nu_1^2}{|\nu|^4} & \frac{-2\nu_1\nu_2}{|\nu|^4}\\ 
  \frac{-2\nu_1\nu_2}{|\nu|^4} & \frac{\nu_1^2-\nu_2^2}{|\nu|^4}
\end{pmatrix}
, H_w K(\tilde{w},y)\right\rangle \\
&=\left\langle
 \begin{pmatrix}
     \nu\left(\frac{-2\nu_1(3\nu_2^2-\nu_1^2)}{|\nu|^6},\frac{2\nu_2(3\nu_1^2-\nu_2^2)}{|\nu|^6} \right) & \nu\left(\frac{-2\nu_2(|\nu|^2-4\nu_1^2)}{ |\nu|^{6}},\frac{-2\nu_1(|\nu|^2-4\nu_2^2)}{ |\nu|^{6}}\right)\\
    \nu\left(\frac{-2\nu_2(|\nu|^2-4\nu_1^2)}{ |\nu|^{6}},\frac{-2\nu_1(|\nu|^2-4\nu_2^2)}{ |\nu|^{6}}\right) & \nu\left(\frac{2\nu_1(3\nu_2^2-\nu_1^2)}{|\nu|^6}, \frac{-2\nu_2(3\nu_1^2-\nu_2^2)}{|\nu|^6} \right)
 \end{pmatrix}
 , H_w K(\tilde{w},y)\right\rangle \\
 &=2|\nu|^{-4} \big( (\nu_1^2-\nu_2^2) \frac{\partial^2}{\partial w_1^2}K(\tilde{w},y)+(\nu_2^2-\nu_1^2)\frac{\partial^2}{\partial w_2^2} K(\tilde{w},y)+4\nu_1\nu_2\frac{\partial^2}{\partial w_1 \partial w_2}K(\tilde{w},y) \big)\\
 &=2|\nu|^{-4} \bigg( 2\nu_1\left(\nu_1\frac{\partial^2} {\partial w_1^2}K(\tilde{w},y)+ \nu_2 \frac{\partial^2}{\partial w_2\partial w_1} K(\tilde{w},y)\right)\\
&\quad+ 2\nu_2\left(\nu_2\frac{\partial^2} {\partial w_2^2}K(\tilde{w},y)+ \nu_1 \frac{\partial^2}{\partial w_1\partial w_2} K(\tilde{w},y)\right) -|\nu|^2\Delta_w K(\tilde{w},y)\bigg)\\
&=4|\nu|^{-4}\nu \left(\nu_1\frac{\partial^2} {\partial w_1^2}K(\tilde{w},y)+ \nu_2 \frac{\partial^2}{\partial w_2\partial w_1} K(\tilde{w},y),\nu_2\frac{\partial^2} {\partial w_2^2}K(\tilde{w},y)+ \nu_1 \frac{\partial^2}{\partial w_1\partial w_2} K(\tilde{w},y)\right).
\end{aligned}
\end{equation}
By substituting \eqref{exterior_derivative_nabla} and \eqref{exterior_derivative_hessian} into \eqref{eq4.5} and noting that $\nu=x$, we obtain
\begin{equation}\label{eq:partial_nu_q_epsilon}
    \begin{aligned}
         \frac{\partial}{\partial \nu}q_\varepsilon  (x,y)&=\nu \bigg( \nabla_x K(x,y) -h(\eps)(2\pi)^{-1}  |x|^{-2} \nabla_w K(\tilde{w},y) \\
&\quad-i(\eps)\frac{2}{\pi} |x|^{-4} \left(x_1\frac{\partial^2} {\partial w_1^2}K(\tilde{w},y)+ x_2 \frac{\partial^2}{\partial w_2\partial w_1} K(\tilde{w},y), \right.\\
&\quad\quad\quad\quad\quad\quad\quad  \left.x_2\frac{\partial^2} {\partial w_2^2}K(\tilde{w},y)+ x_1 \frac{\partial^2}{\partial w_1\partial w_2} K(\tilde{w},y)\right)\\
&\quad+ h(\eps)  \nabla_x\left\langle  \nabla_w S(x,\tilde{w}), \nabla_w K(\tilde{w},y)\right\rangle \\
&\quad+ i(\eps)  \nabla_x\left\langle  H_w S(x,\tilde{w}), H_w K(\tilde{w},y)\right\rangle \bigg).
    \end{aligned}
\end{equation}
Without loss of generality, we can assume that $x=x^*$. Taking into account that $h(\eps)=2\pi |x^*|^{2}$ and $i(\eps)=\frac{\pi}{2} |x^*|^{4}$, we find that \eqref{eq:partial_nu_q_epsilon} is equivalent to
\begin{equation} \label{eq:derivative_at_x_star}
    \begin{aligned} 
        \left. \frac{\partial}{\partial \nu} q_\varepsilon (x, y) \right|_{x = x^*} &= \nu \bigg( \nabla_x K(x^*, y) - \nabla_w K(\tilde{w}, y) \\
        &\quad - \left( x^*_1 \frac{\partial^2}{\partial w_1^2} K(\tilde{w}, y) + x^*_2 \frac{\partial^2}{\partial w_2 \partial w_1} K(\tilde{w}, y), \right. \\
        &\quad\quad \left. x^*_2 \frac{\partial^2}{\partial w_2^2} K(\tilde{w}, y) + x^*_1 \frac{\partial^2}{\partial w_1 \partial w_2} K(\tilde{w}, y) \right) \\
        &\quad + \mathcal{O}(\varepsilon^2) \nabla_x \left\langle \nabla_w S(x^*, \tilde{w}), \nabla_w K(\tilde{w}, y) \right\rangle \\
        &\quad + \mathcal{O}(\varepsilon^4) \nabla_x \left\langle H_w S(x^*, \tilde{w}), H_w K(\tilde{w}, y) \right\rangle \bigg).
    \end{aligned}
\end{equation}

We present two important theorems supporting the proof of Theorem \ref{theo2.1}. The proofs are presented in the following subsection.
\begin{theorem}\label{theo4.1}
     There exists a constant $C$ independent of $\varepsilon$ such that
    \[ \|(\mathbf{K}_\varepsilon - \mathbf{H}_\varepsilon) (\chi_\varepsilon\varphi_i)\|_{L^2(\Omega_{E_\varepsilon})} \le C \varepsilon^3 |\log \eps|^2.\]
    Here $\chi_\varepsilon$ is the characteristic function of $\bar{\Omega}_{E_\varepsilon}$.
\end{theorem}
\begin{theorem} \label{theo4.2}
  There exists a constant $C$ independent of $\varepsilon$ such that
\begin{eqnarray*}
\|(\chi_\varepsilon \mathbf{\overline{H}}_\varepsilon - \mathbf{H}_\varepsilon \chi_\varepsilon)\varphi_i\|_{L^2(\Omega_{E_\varepsilon})} \le C \varepsilon^3 |\log \eps|^2.
\end{eqnarray*}
\end{theorem}
\subsection{Proof of Theorem 4.1}
We frequently use the following properties of the Green function:
\begin{align}
\left\lvert K(x,y) \right\rvert \le C \lvert \log \lvert x-y \rvert\rvert, \label{eq4.6}\\
\left\lvert\nabla_x K(x,y) \right\rvert \le C \lvert x-y\rvert^{-1}. \label{eq4.7}
\end{align}
These properties imply the following estimates:
\begin{align}
\left\lvert( \mathbf{K}f)(x) \right\rvert \le C \|f\|_{L^p(\Omega)} \quad (p>1), \label{eq4.8}\\
\left\lvert\nabla_x (\mathbf{K}f)(x) \right\rvert \le C \|f\|_{L^p(\Omega)} \quad (p>2). \label{eq4.9}
\end{align}
Now we state and prove the following lemma:
\begin{lemma} \label{lemma4.3}
Let $p \in (2, \infty)$. Then, there exists a constant $C >0$ independent of $\varepsilon$ such that
\begin{align*}
\| \mathbf{H_\varepsilon} - \mathbf{K_\varepsilon}\|_{L^p(\Omega_{E_\varepsilon}) } \le C \varepsilon^{2-(2/p)}\log \eps.
\end{align*}
\end{lemma}
\begin{proof}
We consider an arbitrary function $\tilde{f} \in L^p(\Omega)$ which is zero on $E_\varepsilon$. Then, we put $k_\varepsilon=(\mathbf{H_\varepsilon}- \mathbf{K_\varepsilon})\tilde{f}$ and note that $k_\varepsilon$ satisfies $\Delta k_\varepsilon(x)= 0$ for $x \in \Omega_{E_\varepsilon}$ and $k_\varepsilon(x)=0$ for $x \in \partial \Omega$.\\
Taking $x=x^*$ and using \eqref{eq:derivative_at_x_star}, we obtain the expression
	\begin{equation} \label{eq4.10}
\begin{aligned}
        \left. \frac{\partial}{\partial \nu} k_\varepsilon(x) \right|_{x = x^*} &= \nu \bigg( \nabla(\mathbf{K}\tilde{f})(x^*) - \nabla(\mathbf{K}\tilde{f})(\tilde{w}) \\
        &\quad - \left( x^*_1 \frac{\partial^2}{\partial w_1^2} (\mathbf{K}\tilde{f})(\tilde{w}), x^*_2 \frac{\partial^2}{\partial w_2^2} (\mathbf{K}\tilde{f})(\tilde{w}) \right) \\
        &\quad - \left( x^*_2 \frac{\partial^2}{\partial w_2 \partial w_1} (\mathbf{K}\tilde{f})(\tilde{w}), x^*_1 \frac{\partial^2}{\partial w_1 \partial w_2} (\mathbf{K}\tilde{f})(\tilde{w}) \right) \\
        &\quad + \mathcal{O}(\varepsilon^2) \nabla_x \left\langle \nabla_w S(x^*, \tilde{w}), \nabla (\mathbf{K}\tilde{f})(\tilde{w}) \right\rangle \\
        &\quad + \mathcal{O}(\varepsilon^4) \nabla_x \left\langle H_w S(x^*, \tilde{w}), H_w (\mathbf{K}\tilde{f})(\tilde{w}) \right\rangle \bigg)
    \end{aligned}
	\end{equation}
	By the Sobolev embedding theorem, we have
	\begin{eqnarray} \label{eq4.11}
		\|\mathbf{K}\tilde{f}\|_{C^{1+s}(\Omega)} \le C \|\tilde{f}\|_{L^p(\Omega)}=C \|\tilde{f}\|_{L^p(\Omega_{E_\varepsilon})}
	\end{eqnarray}
    for  $s = 1-2/p$, $2< p < \infty$. Using \eqref{eq4.11}, we have
	\begin{align} \label{eq4.12}
	\left\lvert \frac{\partial}{\partial x_n}(\mathbf{K}\tilde{f})(x^*)  -  \frac{\partial}{\partial w_n} (\mathbf{K}\tilde{f})(\tilde{w}) \right\rvert\le C \varepsilon^s \|\mathbf{K}\tilde{f}\|_{C^{1+s}(\Omega)}  \le C' \varepsilon^s\|\tilde{f}\|_{L^p(\Omega_{E_\varepsilon})}
	\end{align}
	for $n=1,2$, $p>2$.
	
	We define that $R = \sup_{x\in \partial \Omega}|x|$, and  $M_{\min} = \inf_{x\in \partial E}|x|$, then  $\Omega_{E_\eps}$ is in the annulus with centre $\tilde{w}$:
 \[A(\tilde{w};\eps M_{\min},R)=\{x\in \Rr^2 \vert \eps M_{\min} <|x-\tilde{w}|<R \}.\] 
 We obtain the estimates \eqref{eq4.13} and \eqref{eq4.14} as follows:
\begin{eqnarray} \label{eq4.13}
\left\lvert \frac{\partial}{\partial w_n} (\mathbf{K} \tilde{f})(\tilde{w})\right\rvert &\le& C \left( \int_{\Omega_{E_\varepsilon}} \lvert y-\tilde{w}\rvert^{-p'} dy\right)^{1/p'} \|\tilde{f}\|_{L^p(\Omega_{E_\varepsilon})}\\
&\le& C\left(\int_{\eps M_\text{min}}^R\int_{0}^{2\pi} r^{-p'} r dr d\theta\right)^{1/p'} \|\tilde{f}\|_{L^p(\Omega_{E_\varepsilon})} \nonumber\\
&\le& \begin{cases}
    C\lvert \log \varepsilon \rvert^{1/2} \|\tilde{f}\|_{L^2(\Omega_{E_\varepsilon})} & \text{if }p=2 \nonumber\\
 C \|\tilde{f}\|_{L^p(\Omega_{E_\varepsilon})} &\text{if } p>2
\end{cases}
\end{eqnarray}
for $n=1,2$, $p'$ satisfies $(1/p)+(1/p') = 1$.

Additionally, we have
	\begin{eqnarray} \label{eq4.14}
	\left\lvert \frac{\partial^2}{\partial w_m \partial w_n} (\mathbf{K} \tilde{f})(\tilde{w})\right\rvert &\le& C \left( \int_{\Omega_{E_\varepsilon}}\lvert y-\tilde{w}\rvert^{-2p'} dy\right)^{1/p'} \|\tilde{f}\|_{L^p(\Omega_{E_\varepsilon})} \nonumber\\
 &\le&  C\left(\int_{\eps M_\text{min}}^R\int_{0}^{2\pi} r^{-2p'} r dr d\theta\right)^{1/p'} \|\tilde{f}\|_{L^p(\Omega_{E_\varepsilon})} \nonumber\\
 &\le& 
     C  \varepsilon^{-2/p} \|\tilde{f}\|_{L^p(\Omega_{E_\varepsilon})}  \qquad \text{if } p>1 
	\end{eqnarray}
	for $m \ge 1$, $ n \le 2$.
	
	On the other hand, we see that
\begin{eqnarray} \label{eq4.15}
\frac{\partial}{\partial x_n}\left\langle  \nabla_w S(x,\tilde{w}), \nabla(\mathbf{K}\tilde{f})(\tilde{w})\right\rangle &\le& C \left( \int_{\Omega_{E_\varepsilon}} \lvert y-\tilde{w}\rvert^{-p'} dy \right)^{1/p'} \|\tilde{f}\|_{L^p(\Omega_{E_\varepsilon})} \nonumber\\
&\le& C \|\tilde{f}\|_{L^p(\Omega_{E_\varepsilon})},
\end{eqnarray}
\begin{eqnarray} \label{eq4.16}
\frac{\partial}{\partial x_n}\left\langle  H_w S(x,\tilde{w}), H_w (\mathbf{K} \tilde{f})(\tilde{w})\right\rangle &\le& C \left( \int_{\Omega_{E_\varepsilon}} \lvert y-\tilde{w}\rvert^{-2p'} dy\right)^{1/p'} \|\tilde{f}\|_{L^p(\Omega_{E_\varepsilon})} \nonumber\\
	&\le& C  \varepsilon^{-2/p} \|\tilde{f}\|_{L^p(\Omega_{E_\varepsilon})} \qquad \text{ for } n=1,2.
\end{eqnarray}
By combining the estimates \eqref{eq4.12}, \eqref{eq4.14}, \eqref{eq4.15}, and \eqref{eq4.16}, we obtain
 \begin{eqnarray*}
\max_{x \in \partial E_\varepsilon} \left\lvert \frac{\partial}{\partial \nu}k_\varepsilon(x) \right\rvert &\le& C \left(\varepsilon^s + \varepsilon^{1-2/p}+ \varepsilon^2 + \varepsilon^{4-2/p}\right)
\|\tilde{f}\|_{L^p(\Omega_{E_\varepsilon})}\\
 &\le& C \varepsilon^{1-2/p}
\|f\|_{L^p(\Omega_{E_\varepsilon})}
\end{eqnarray*}
for $p>2$.

Using Lemma \ref{lemma3.3}, we can conclude that
\begin{eqnarray*}
\|\mathbf{H_\varepsilon} - \mathbf{K_\varepsilon}\|_{L^p(\Omega_{E_\varepsilon})} \le C \varepsilon^{2-2/p} |\log\eps|,
\end{eqnarray*}
which completes the proof of Lemma \ref{lemma4.3}.
\end{proof}
In the following proposition, we establish the convergence rate of the difference between operators $\mathbf{H_\varepsilon}$ and $\mathbf{K_\varepsilon}$ in the $L^p$ norm, providing an estimate for their approximation error for $p \in (1, \infty]$.
\begin{proposition}\label{proposition4.4}
     Fix $p \in (1, \infty]$. Then, for any fixed $s \in (0,1)$, the operator norm $$\|\mathbf{H_\varepsilon} - \mathbf{K_\varepsilon}\|_{L^p(\Omega_{E_\varepsilon}) } = \mathcal{O}(\varepsilon^{1+s})$$ as $\varepsilon$ tends to zero.
\end{proposition} 
\begin{proof}
	Assuming $p \in (1, \infty)$, let $\mathbf{Q_\varepsilon} = \mathbf{H_\varepsilon} - \mathbf{K_\varepsilon}$. Since $\mathbf{Q_\varepsilon}$ is self-adjoint on $L^2(\Omega_{E_\varepsilon})$, it follows that $\|\mathbf{Q_\varepsilon}\|_{L^q(\Omega_{E_\varepsilon})} = \|\mathbf{Q_\varepsilon}\|_{L^{q'}(\Omega_{E_\varepsilon})}$, where $\frac{1}{q} + \frac{1}{q'} = 1$.

By applying the Riesz-Thorin interpolation theorem, we can conclude that $\|\mathbf{Q}_\varepsilon\|_{L^p(\Omega_{E_\varepsilon})} \le \|\mathbf{Q}_\varepsilon\|_{L^q(\Omega_{E_\varepsilon})}$ for any $p \in (q',q)$, where $q > 2$. By choosing a sufficiently large $q$ and utilizing Lemma \ref{lemma4.3}, we obtain Proposition \ref{proposition4.4} for $p \neq 1, \infty$.

Now, consider the case when $p = \infty$. Using a similar argument as in the proof of Lemma \ref{lemma4.3}, we can establish Proposition \ref{proposition4.4} with $p = \infty$.

Therefore, by combining the results for $p \in (1, \infty)$ and $p = \infty$, we conclude the validity of Proposition \ref{proposition4.4} for all $p \in (1, \infty]$.
\end{proof}

Next, we aim to estimate $\|(\mathbf{K}_\varepsilon - \mathbf{H}_\varepsilon) (\chi_\varepsilon\varphi_i)\|_{L^2(\Omega_{E_\varepsilon})}$. Let us denote $ v_\varepsilon= (\mathbf{K}_\varepsilon - \mathbf{H}_\varepsilon) (\chi_\varepsilon\varphi_i)$. Put $\hat{\chi}_\varepsilon= 1- \chi_\varepsilon$. By using equation (\ref{eq4.10}), we obtain the following expression 
\begin{align}\label{eq4.17}
    \left. \frac{\partial}{\partial \nu}  v_\varepsilon(x)\right\rvert_{ x=x^*} = \nu \left(I_1(\varepsilon) - I_2(\varepsilon) + I_3(\varepsilon) \right),
\end{align}
		where
		\begin{eqnarray*}
		I_1(\varepsilon)&=&  \nabla(\mathbf{K}\varphi_i)(x^*) -  \nabla(\mathbf{K}\varphi_i)(\tilde{w}) - \left(x^*_1\frac{\partial^2}{\partial w_1^2} (\mathbf{K}\varphi_i)(\tilde{w}),x^*_2\frac{\partial^2}{\partial w_2^2} (\mathbf{K}\varphi_i)(\tilde{w}) \right) \\
      &&- \left(x^*_2\frac{\partial^2}{\partial w_2 \partial w_1} (\mathbf{K}\varphi_i)(\tilde{w}),x^*_1\frac{\partial^2}{\partial w_1 \partial w_2} (\mathbf{K}\varphi_i)(\tilde{w}) \right)  \\
		I_2(\varepsilon)&=& \nabla(\mathbf{K}\hat{\chi}_\eps\varphi_i)(x^*) -  \nabla(\mathbf{K}\hat{\chi}_\eps\varphi_i)(\tilde{w}) - \left(x^*_1\frac{\partial^2}{\partial w_1^2} (\mathbf{K}\hat{\chi}_\eps\varphi_i)(\tilde{w}),x^*_2\frac{\partial^2}{\partial w_2^2} (\mathbf{K}\hat{\chi}_\eps\varphi_i)(\tilde{w}) \right) \\
      &&- \left(x^*_2\frac{\partial^2}{\partial w_2 \partial w_1} (\mathbf{K}\hat{\chi}_\eps\varphi_i)(\tilde{w}),x^*_1\frac{\partial^2}{\partial w_1 \partial w_2} (\mathbf{K}\hat{\chi}_\eps\varphi_i)(\tilde{w}) \right) \\
I_3(\varepsilon) &=&\mathcal{O}(\eps^2) \nabla_x\left\langle  \nabla_w S(x,\tilde{w}), \nabla_w (\mathbf{K} \chi_\varepsilon\varphi_i)(\tilde{w})\right\rangle \nonumber \\
&&+ \mathcal{O}(\eps^4) \nabla_x\left\langle  H_w S(x^*,\tilde{w}), H_w (\mathbf{K}\chi_\varepsilon\varphi_i)(\tilde{w})\right\rangle.
	\end{eqnarray*}
	We can express $I_1(\varepsilon)$ as follows, taking into account that $\mathbf{K}\varphi_i(x) = \mu_i^{-1} \varphi_i(x)$:
\begin{align*}
I_1(\varepsilon) =& \mu_i^{-1}\bigg(  \nabla\varphi_i(x^*) -  \nabla\varphi_i(\tilde{w}) - \left(x^*_1\frac{\partial^2}{\partial w_1^2} (\varphi_i)(\tilde{w}),x^*_2\frac{\partial^2}{\partial w_2^2} (\varphi_i)(\tilde{w}) \right) \\
      &- \left(x^*_2\frac{\partial^2}{\partial w_2 \partial w_1} (\varphi_i)(\tilde{w}),x^*_1\frac{\partial^2}{\partial w_1 \partial w_2} (\varphi_i)(\tilde{w}) \right)\bigg)
\end{align*}
By using the Taylor expansion at $\tilde{w}=(0,0)$ for $\frac{\partial}{\partial x_n} \varphi_i(x^*)$, $n=1,2$ we obtain:
\begin{align*}
\frac{\partial}{\partial x_1} \varphi_i(x^*) = \frac{\partial}{\partial w_1}\varphi_i(\tilde{w}) + x^*_1\frac{\partial^2}{{\partial w_1}^2} \varphi_i(\tilde{w}) +x^*_2\frac{\partial^2}{\partial w_2\partial w_1} (\varphi_i)(\tilde{w})+ \mathcal{O}(\eps^2),\\
\frac{\partial}{\partial x_2} \varphi_i(x^*) = \frac{\partial}{\partial w_2}\varphi_i(\tilde{w}) + x^*_2\frac{\partial^2}{{\partial w_2}^2} \varphi_i(\tilde{w}) +x^*_1\frac{\partial^2}{\partial w_1\partial w_2} (\varphi_i)(\tilde{w})+ \mathcal{O}(\eps^2),
\end{align*}
 which implies that
	\begin{eqnarray} \label{eq4.18}
	\lvert I_1(\varepsilon)\rvert \le C \varepsilon^2.
	\end{eqnarray}
 Furthermore, we have the following estimate:
	\begin{eqnarray} \label{eq4.19}
	\lvert I_3(\varepsilon)\rvert &\le& C (\varepsilon^2 + \varepsilon^4 \varepsilon^{-2/p}) \le C \varepsilon^2.
	\end{eqnarray}
 To estimate $I_2(\varepsilon)$, we introduce the function $L(x,y)= -(2\pi)^{-1}\log \lvert x-y \rvert$. Then, we can express $I_2(\varepsilon)$ as:
		\begin{eqnarray} \label{eq4.20}
			I_2(\varepsilon) = \left(I_4(\varepsilon)+ I_5(\varepsilon)+ I_6(\varepsilon), I'_4(\varepsilon)+ I'_5(\varepsilon)+ I'_6(\varepsilon)\right),
		\end{eqnarray}
	where
	\begin{eqnarray*}
	I_4(\varepsilon) &=& \frac{\partial}{\partial x_1} \left( \int_{E_\varepsilon} L(x^*,y)(\varphi_i(y)-\varphi_i(x^*))dy\right) \\
	&& - \frac{\partial}{\partial w_1} \int_{E_\varepsilon} L(w,y)(\varphi_i(y)-\varphi_i(w))dy \\
	&&- (x^*_1-w_1) \frac{\partial^2}{{\partial w_1}^2} \int_{E_\varepsilon} L(w,y)(\varphi_i(y) -\varphi_i(w)- \sum_{n=1}^{2}(y_n-w_n)\frac{\partial \varphi_i}{\partial w_n}(w)) dy \\
 &&- (x^*_2-w_2) \frac{\partial^2}{\partial w_2\partial w_1} \int_{E_\varepsilon} L(w,y)(\varphi_i(y) -\varphi_i(w)- \sum_{n=1}^{2}(y_n-w_n)\frac{\partial \varphi_i}{\partial w_n}(w)) dy \\
	I_5(\varepsilon) &=& \frac{\partial}{\partial x_1} (\varphi_i(x^*) F(x^*)) -\left(\frac{\partial}{\partial w_1}+ (x^*_1-w_1) \frac{\partial^2}{ {\partial w_1}^2 }\right) (\varphi_i(w)F(w)) \\
	&&- (x^*_1-w_1)\sum_{n=1}^{2} \frac{\partial^2}{{\partial w_1}^2}\left(\frac{\partial \varphi_i}{\partial w_n} (w) K_n(w) \right) - (x^*_2-w_2)\sum_{n=1}^{2} \frac{\partial^2}{{\partial w_2\partial w_1}}\left(\frac{\partial \varphi_i}{\partial w_n} (w) K_n(w) \right)  \\
	I_6(\varepsilon) &=& \frac{\partial}{\partial x_1} (\mathbf{S}\hat{\chi}_\varepsilon \varphi_i)(x^*) - \left(\frac{\partial}{\partial w_1} + (x^*_1-w_1) \frac{\partial^2}{{\partial w_1}^2} + (x^*_2-w_2) \frac{\partial^2}{\partial w_2\partial w_1} \right) (\mathbf{S} \hat{\chi}_\varepsilon \varphi_i)(w),
	\end{eqnarray*} 
 \begin{eqnarray*}
	I'_4(\varepsilon) &=& \frac{\partial}{\partial x_2} \left( \int_{E_\varepsilon} L(x^*,y)(\varphi_i(y)-\varphi_i(x^*))dy\right) \\
	&& - \frac{\partial}{\partial w_2} \int_{E_\varepsilon} L(w,y)(\varphi_i(y)-\varphi_i(w))dy \\
	&&- (x^*_2-w_2) \frac{\partial^2}{{\partial w_1}^2} \int_{E_\varepsilon} L(w,y)(\varphi_i(y) -\varphi_i(w)- \sum_{n=1}^{2}(y_n-w_n)\frac{\partial \varphi_i}{\partial w_n}(w)) dy \\
 &&- (x^*_1-w_1) \frac{\partial^2}{\partial w_2\partial w_1} \int_{E_\varepsilon} L(w,y)(\varphi_i(y) -\varphi_i(w)- \sum_{n=1}^{2}(y_n-w_n)\frac{\partial \varphi_i}{\partial w_n}(w)) dy \\
	I'_5(\varepsilon) &=& \frac{\partial}{\partial x_2} (\varphi_i(x^*) F(x)) -\left(\frac{\partial}{\partial w_2}+ (x^*_2-w_2) \frac{\partial^2}{ {\partial w_2}^2 }\right) (\varphi_i(w)F(w)) \\
	&&- (x^*_2-w_2)\sum_{n=1}^{2} \frac{\partial^2}{{\partial w_2}^2}\left(\frac{\partial \varphi_i}{\partial w_n} (w) K_n(w) \right) - (x^*_1-w_1)\sum_{n=1}^{2} \frac{\partial^2}{{\partial w_2\partial w_1}}\left(\frac{\partial \varphi_i}{\partial w_n} (w) K_n(w) \right)  \\
	I'_6(\varepsilon) &=& \frac{\partial}{\partial x_2} (\mathbf{S}\hat{\chi}_\varepsilon \varphi_i)(x^*) - \left(\frac{\partial}{\partial w_2} + (x^*_2-w_2) \frac{\partial^2}{{\partial w_2}^2} + (x^*_1-w_1) \frac{\partial^2}{\partial w_2\partial w_1} \right) (\mathbf{S} \hat{\chi}_\varepsilon \varphi_i)(w),
	\end{eqnarray*}
 for $w=\tilde{w}=0$. In the given expression, the operator $\mathbf{S}$ and functions $F(x)$ and $K_n(w)$ are defined as follows:
	\begin{eqnarray*}
	(\mathbf{S}f)(x) &=& \int_\Omega S(x,y) f(y)dy \\
	F(x) &=& \int_{E_\varepsilon} L(x,y)dy \\
	K_n(w) &=& \int_{E_\varepsilon} L(w,y)(y_n-w_n)dy \qquad (n=1,2).
	\end{eqnarray*}
	Since $(\mathbf{S}f)(x)$ is a smooth function on $\Omega$, we can estimate $I_6(\varepsilon)$ as follows:
	\begin{eqnarray} \label{eq4.21}
	\lvert I_6(\varepsilon) \rvert \le C \varepsilon^2.
	\end{eqnarray}
  We evaluate the partial derivative of the integral $ \frac{\partial}{\partial x_1}\int_{E_\varepsilon} L(x,y)(\varphi_i(y)-\varphi_i(x))dy$  at the point $x=x^*$. Using the properties of eigenfunctions, we can express $\varphi_i(y)$ as $\varphi_i(y) = \varphi_i(x) + \mathcal{O}(\lvert x-y\rvert)$. Substituting this into the expression, we have:
\begin{align*}
&\frac{\partial}{\partial x_1} \int_{E_\varepsilon} L(x,y)(\varphi_i(y)-\varphi_i(x))dy\bigg\vert_{x=x^*} \\
&= -(2\pi)^{-1}\int_{E_\varepsilon} \left( \frac{\partial}{\partial x_1} \log \lvert x-y \rvert(\varphi_i(y)-\varphi_i(x)) +  \log \lvert x-y \rvert\frac{\partial}{\partial x_1} \varphi_i(x) \right)dy \bigg\vert_{x=x^*}  \\
&=   -(2\pi)^{-1} \int_{E_{\varepsilon}} \left( \frac{y_1-x_1}{\lvert x-y \rvert^2} \mathcal{O}( \lvert x-y\rvert) + \frac{\partial}{\partial x_1}\varphi_i(x)  \log \lvert x-y\rvert \right) dy\bigg\vert_{x=x^*} \\
&\le C  \int_{E_{\varepsilon}} (1+ \log \lvert x-y\rvert) dy\bigg\vert_{x=x^*}  \\
& \le C\int_{E_{\varepsilon}} \log \lvert x-y\rvert dy\bigg\vert_{x=x^*}\\
&\le C\int_{\mathbb{B}(x, 2\varepsilon M_{\text{max}})} \log \lvert x-y\rvert dy\bigg\vert_{x=x^*}\\
& \le C \varepsilon^2 \lvert \log \varepsilon\rvert,
\end{align*}
where $C$ is a constant independent of $\varepsilon$ and $M_{\max} = \sup_{x\in\partial E}|x|$. Note that $E_\varepsilon \subset \mathbb{B}(x,2\varepsilon M_{\text{max}})$, which allows us to replace the integral over $E_\varepsilon$ with the integral over the larger ball.

Similarly, we can also establish the following inequalities:
\begin{align*}
&\frac{\partial}{\partial w_1} \int_{E_\varepsilon} L(w,y)(\varphi_i(y)-\varphi_i(w))dy\leq C \int_{E_{\varepsilon}} L(\tilde{w},y) dy, \\
&\frac{\partial^2}{\partial w_1^2} \int_{E_\varepsilon} L(w,y)\left(\varphi_i(y) -\varphi_i(w)- \sum_{n=1}^{2}(y_n-w_n)\frac{\partial \varphi_i}{\partial w_n}(w)\right) dy \leq C \int_{E_{\varepsilon}} L(\tilde{w},y) dy.
\end{align*}
   Combining these estimates, we can derive the following inequality:
\begin{equation}
    \begin{aligned} \label{eq4.22}
\lvert I_4(\varepsilon) \rvert &\leq C \int_{E_\varepsilon} \log \lvert x-y \rvert dy \bigg\vert_{x=x^*} + C \int_{E_\varepsilon} \log \lvert\tilde{w}-y\rvert dy \\
&\qquad+ C \varepsilon \int_{E_\varepsilon} \log \lvert \tilde{w}-y\rvert dy \\
&\leq C \varepsilon^2 \lvert \log \varepsilon \rvert.
\end{aligned}
\end{equation}

After performing the necessary calculations, we arrive at the desired results. 
More precisely, by the mean value theorem, there exist points $c\in E_\eps$, $d\in \mathbb{B}_{M\eps}$, where $M$ is given by \eqref{HaiThamGraz} such that
\begin{align*}
    \int_{E_\eps}L(x,y)dy= \frac{-1}{2\pi} \log|x-c| |E_\eps|,\\
    \int_{\mathbb{B}_{M\eps}}L(x,y)dy= \frac{-1}{2\pi} \log|x-d| |\mathbb{B}_{M\eps}|.
\end{align*}
Since $|\mathbb{B}_{M\eps}|= |E_\eps|$, we can deduce that
\begin{align*}
    F(x)= \frac{\log|x-c|}{\log|x-d|}  \int_{\mathbb{B}_{M\eps}}L(x,y)dy,
\end{align*}
or 
	\begin{align} 
	F(x) = 
 \begin{cases}\label{eq4.23}
     \frac{\log|x-c|}{\log|x-d|}\left(- \frac{1}{2} (M\varepsilon)^2 \lvert \log |x|\rvert\right)  \quad \text{for } x \in  \mathbb{R}^2\backslash \overline{\mathbb{B}}_{M\varepsilon} \\
	  \frac{\log|x-c|}{\log|x-d|}\left(-\frac{1}{2} (M\varepsilon)^2\lvert \log M\varepsilon\rvert + \frac{(M\varepsilon)^2 }{4} -\frac{\lvert x \rvert^2}{4} \right)
	 \quad  \text{for } x \in \mathbb{B}_{M\varepsilon}, 
 \end{cases}
	\end{align}
 and for $K_n(w)$, we have:
\begin{align} \label{eq4.24}
	K_n(w) =   \frac{\log|w-c|}{\log|w-d|}w_n \left( \frac{1}{2}(M\varepsilon)^2   \lvert \log (M\varepsilon)\rvert - \frac{(M\varepsilon)^2}{4}  +\frac{\lvert w \rvert^2}{4}\right) \quad  \text{for } w \in \mathbb{B}_\varepsilon.
	\end{align}
	Since $x^* \in \partial\mathbb{B}_{M\eps}$, we have
	\begin{align} \label{eq4.25}
	F(x^*) &= \frac{\log|x^*-c|}{\log|x^*-d|}\left(-\frac{1}{2}  (M\varepsilon)^2 \lvert \log (M\varepsilon)\rvert\right), \nonumber\\
	\partial F(x^*)/\partial x_1 &= \frac{\partial}{\partial x_1}\frac{\log|x^*-c|}{\log |x^*-d|} \left( \frac{-1}{2}(M\eps)^2|\log(M\eps)|\right)+\frac{\log|x^*-c|}{\log |x^*-d|}\left(-\frac{1}{2} x^*_1\right).
	\end{align}
   For $w=(0,0)$, we obtain
	\begin{align} \label{eq4.26}
 \begin{split}
     F(w) &=\frac{\log c}{\log d}\left(-\frac{1}{2} (M\varepsilon)^2 \lvert \log ( M\varepsilon)\rvert + \frac{ (M\varepsilon)^2 }{4} \right)\\
	\partial F(w)/\partial w_1 &= \frac{\partial}{\partial w_1}\frac{\log|w-c|}{\log |w-d|}\Big|_{w=(0,0)} \left( \frac{-1}{2}(M\eps)^2|\log(M\eps)| +\frac{(M\eps)^2}{4} \right), \\
	\partial^2F(w)/ \partial w_1^2 &= \frac{\partial^2}{\partial w_1^2}\frac{\log|w-c|}{\log |w-d|}\Big|_{w=(0,0)} \left( \frac{-1}{2}(M\eps)^2|\log(M\eps)| +\frac{(M\eps)^2}{4} \right) -\frac{1}{2}\frac{\log c}{\log d}.
 \end{split}
	\end{align}
 Additionally, we have
	\begin{equation} \label{eq4.27}
 \begin{split}
     K_n(w) &= \partial^2 K_n(w)/\partial w_1^2 = 0, \\
	\partial K_n(w)/\partial w_1 &= \delta_{1,n} \frac{\log c}{\log d}\left( \frac{1}{2}(M\varepsilon)^2  \lvert \log ( M\varepsilon)\rvert - \frac{1}{4}(M\varepsilon)^2\right)
 \end{split}
 \end{equation}
where $\delta_{1,n}$ is the Kronecker delta. 
Summing up these results, we obtain:
	 \begin{equation}\label{eq4.28}
	     \begin{aligned}
	          I_5(\varepsilon) &= - \frac{1}{2} x^*_1 \left( \varphi(x^*)\frac{\log|x^*-c|}{\log |x^*-d|} - \varphi(w) \frac{\log|w-c|}{\log |w-d|} \right) + \mathcal{O}(\varepsilon^2 \lvert \log \varepsilon\rvert) \\
	 &=  \mathcal{O}(\varepsilon^2 \lvert \log \varepsilon \rvert).
	     \end{aligned}
	 \end{equation} 
Using estimates \eqref{eq4.17}, \eqref{eq4.18}, \eqref{eq4.19}, \eqref{eq4.20}, \eqref{eq4.21}, \eqref{eq4.22} and \eqref{eq4.28}, we conclude that: 
	 $$ \left\lvert\frac{\partial}{\partial x_1}  v_\varepsilon(x)\right\rvert_{ x=x^*} \le C \varepsilon^2 \lvert\log \varepsilon\rvert. $$
In the same way estimating $I_4(\varepsilon)$, $I_5(\varepsilon)$ and $I_6(\varepsilon)$, we can carry out $I'_4(\varepsilon)$, $I'_5(\varepsilon)$ and $I'_6(\varepsilon)$ and obtain the same results, which yields
    $$ \left\lvert\frac{\partial}{\partial x_2}  v_\varepsilon(x)\right\rvert_{ x=x^*} \le C \varepsilon^2 \lvert\log \varepsilon\rvert. $$
Combining these facts, we derive 
$$ \left| \frac{\partial}{\partial \nu}  v_\varepsilon(x)\right\rvert_{ x=x^*} \le C \varepsilon^2 \lvert\log \varepsilon\rvert. $$
	 By Lemma \ref{lemma3.3}, we can further deduce that
	 \[ \|v_\varepsilon\|_{L^2(\Omega_{E_\varepsilon})} \le C \varepsilon^3 |\log \eps|^2.\] 
	 Therefore, we obtain Theorem \ref{theo4.1}.
\subsection{Convergence of eigenvalues}
In the context of analyzing the convergence of eigenvalues, we introduce a new kernel $\tilde{h}(x,y)$ defined as follows:
\begin{equation}\label{eq4.29}
\begin{split}
     \tilde{h}_\varepsilon(x,y) =K(x,y) &+ h(\varepsilon)\left\langle \nabla_w  K(x,\tilde{w}), \nabla_w K(\tilde{w},y)\right\rangle \chi_\varepsilon(x)\chi_\varepsilon(y) \\
&+  i(\varepsilon)\left\langle H_w  K(x,\tilde{w}), H_w K(\tilde{w},y)\right\rangle \chi_\varepsilon(x)\chi_\varepsilon(y). 
\end{split}
\end{equation}
Here, $K(x,y)$ is a given kernel, and $\nabla_w K(x,\tilde{w})$ and $H_w K(x,\tilde{w})$ represent the gradient and Hessian of $K$ with respect to the variable $w$, respectively. The functions $h(\varepsilon)$ and $i(\varepsilon)$ are introduced as weight functions that depend on a small parameter $\varepsilon$, and $\chi_\varepsilon(x)$ is a characteristic function defined as $\chi_\varepsilon(x) = 1$ if $x \in \bar{\Omega}_{E_\varepsilon}$ and $\chi_\varepsilon(x) = 0$ otherwise.

Using the kernel $\tilde{h}_\varepsilon(x,y)$, we define the operator $\mathbf{\tilde{H}}_\varepsilon$ as:
\begin{align*}
    (\mathbf{\tilde{H}}_\varepsilon f) (x) = \int _{\Omega} \tilde{h}_\varepsilon (x,y) f(y) dy.
\end{align*}
\begin{lemma}\label{lemma4.5}
  There exists a constant $C$ independent of $\varepsilon$ such that
    \begin{align}\label{eq4.30}
    \| \mathbf{\tilde{H}}_\varepsilon - \chi_\varepsilon \mathbf{\tilde{H}}_\varepsilon \chi_\varepsilon \|_{L^2(\Omega)} \le C \varepsilon
   \end{align}
   holds.
\end{lemma}
\begin{proof}
We begin by considering the operator $\mathbf{\tilde{H}}_\varepsilon - \chi_\varepsilon \mathbf{\tilde{H}}_\varepsilon \chi_\varepsilon$. Using the expression
\begin{align*}
        \mathbf{\tilde{H}}_\varepsilon - \chi_\varepsilon \mathbf{\tilde{H}}_\varepsilon \chi_\varepsilon = (1-\chi_\varepsilon)\mathbf{\tilde{H}}_\varepsilon \chi_\varepsilon + \mathbf{\tilde{H}}_\varepsilon(1-\chi_\varepsilon),
    \end{align*}
    we can estimate its norm as follows:
    \begin{align}\label{eq4.31}
          \| \mathbf{\tilde{H}}_\varepsilon - \chi_\varepsilon \mathbf{\tilde{H}}_\varepsilon \chi_\varepsilon \|_{L^2(\Omega)} \le  \|(1-\chi_\varepsilon)\mathbf{\tilde{H}}_\varepsilon\|_{L^2(\Omega)} + \| \mathbf{\tilde{H}}_\varepsilon(1-\chi_\varepsilon)\|_{L^2(\Omega)}.
    \end{align}
    Next, we observe that in the $h(\varepsilon)$ and $i(\varepsilon)$ terms of \eqref{eq4.29}, the product $(1-\chi_\varepsilon)\chi_\varepsilon$ equals zero. Therefore, we have the estimate:
    \begin{align*}
        \|(1-\chi_\varepsilon)\mathbf{\tilde{H}}_\varepsilon v\|_{L^2(\Omega)}  \le C \lvert E_\varepsilon \rvert^{1/2} \|\mathbf{K}v\|_{L^2(\Omega)} \le C \varepsilon \|v\|_{L^2(\Omega)},
    \end{align*}
    which holds for any $v\in L^2(\Omega)$. Consequently, we obtain the following inequalities:
    \begin{equation}\label{eq4.32}
        \begin{split}
             &\|(1-\chi_\varepsilon)\mathbf{\tilde{H}}_\varepsilon \|_{L^2(\Omega)}  \le C  \varepsilon\\
             &\|(1-\chi_\varepsilon)\mathbf{\tilde{H}}_\varepsilon \chi_\varepsilon\|_{L^2(\Omega)}  \le C  \varepsilon.
        \end{split}
    \end{equation}
   Since we have a duality relation
    \begin{align*}
        ((1-\chi_\varepsilon)\mathbf{\tilde{H}}_\varepsilon)^\ast = \mathbf{\tilde{H}}_\varepsilon(1-\chi_\varepsilon),
    \end{align*}
    we can further deduce the following: 
    \begin{align}\label{eq4.33}
        \|\mathbf{\tilde{H}}_\varepsilon(1-\chi_\varepsilon)\|_{L^2(\Omega)} \le C \varepsilon.
    \end{align}
    We have completed the proof of the lemma by combining the estimates \eqref{eq4.31} and \eqref{eq4.32}.
\end{proof}
Next, we aim to estimate $\|\mathbf{\tilde{H}}_\varepsilon - \mathbf{K}\|_{L^2(\Omega)}$. For this purpose, we state and prove the following lemma:
\begin{lemma}\label{lemma4.6}
    There exists a constant $C$ independent of $\varepsilon$ such that
    \begin{align*}
         \| \mathbf{\tilde{H}}_\varepsilon -\mathbf{K}\|_{L^2(\Omega)} \le C \varepsilon^2 \lvert \log \varepsilon\rvert
    \end{align*}
    holds.
\end{lemma}
\begin{proof}
We consider an arbitrary function $v \in L^p(\Omega)$. From the expression
\begin{align*}
((\mathbf{\tilde{H}}_\varepsilon- \mathbf{K})v)(x) = &h(\varepsilon)\langle \nabla_w K(x,w), \nabla_w (\mathbf{K}\xi_\varepsilon v)(\tilde{w}) \rangle  \chi_\varepsilon(x) \\
+&i(\varepsilon) \langle H_w K(x,w), H_w (\mathbf{K}\xi_\varepsilon v) (\tilde{w}) \rangle  \chi_\varepsilon(x), 
\end{align*}
we can estimate the norm as follows:
\begin{multline}\label{eq4.34}
    \| (\mathbf{\tilde{H}}_\varepsilon -\mathbf{K})v\|_{L^p(\Omega)}  \le \lvert h(\varepsilon)\rvert \sum_{n=1}^2 \left( \int_{\Omega_{E_\varepsilon}} \left \lvert  \frac{\partial}{\partial w_n} K(x,\tilde{w})\right\rvert^p \,dx\right)^{1/p} \left\lvert \frac{\partial}{\partial w_n} (\mathbf{K} \chi_\varepsilon v)(\tilde{w})\right\rvert \\
    +\lvert i(\varepsilon)\rvert \sum_{m,n=1}^2 \left( \int_{\Omega_{E_\varepsilon}} \left \lvert  \frac{\partial^2}{\partial w_m\partial w_n} K(x,\tilde{w})\right\rvert^p \,dx\right)^{1/p} \left\lvert \frac{\partial^2}{\partial w_m \partial w_n} (\mathbf{K} \chi_\varepsilon v)(\tilde{w})\right\rvert
\end{multline}
valid for $p<1$.

We use the following inequalities to estimate the terms in \eqref{eq4.33}. First, for $n=1,2$, we have
\begin{equation}\label{eq4.35}
\begin{aligned}
   \left( \int_{\Omega_{E_\varepsilon}} \left \lvert \frac{\partial}{\partial w_n} K(x,\tilde{w})\right\rvert^p \,dx\right)^{1/p} &\le C \left( \int_{\Omega_{E_\varepsilon}} \lvert x-\tilde{w}\rvert^{-p} ,dx\right)^{1/p} \\
&\le 
\begin{cases}
C \lvert \log \varepsilon\rvert^{1/2}  &\text{ if } p=2, \\
C \varepsilon^{2/p-1}  &\text{ if } p>2.
\end{cases}
    \end{aligned}
\end{equation}
Next, for $1 \le m,n \le 2$, we have
\begin{equation}\label{eq4.36}
\begin{aligned}
     \left( \int_{\Omega_{E_\varepsilon}} \left \lvert  \frac{\partial^2}{\partial w_m\partial w_n} K(x,\tilde{w})\right\rvert^p \,dx\right)^{1/p}  &\le C \left( \int_{\Omega_{E_\varepsilon}} \lvert x-\tilde{w}\rvert^{-2p} \,dx\rvert\right)^{1/p} \\
    &\le C \varepsilon^{2/p-2} \qquad\text{ if } p>1.
\end{aligned}
\end{equation}
By \eqref{eq4.34}, \eqref{eq4.35}, \eqref{eq4.36} and using the estimation \eqref{eq4.13}, \eqref{eq4.14} with $\tilde{f}= \chi_\varepsilon v$,  we can conclude that
\begin{align*}
     \| (\mathbf{\tilde{H}}_\varepsilon -\mathbf{K})v\|_{L^2(\Omega)} &\le C( \lvert h(\varepsilon)\rvert \lvert \log \varepsilon\rvert^{1/2} \lvert \log \varepsilon\rvert^{1/2} \|v\|_{L^2(\Omega)} \\
     &\quad + \lvert i(\varepsilon)\rvert \varepsilon^{-1} \varepsilon^{-1}\|v\|_{L^2(\Omega)} )\\
     &\le C \varepsilon^2 \lvert \log \varepsilon\rvert \|v\|_{L^2(\Omega)},
\end{align*}
which holds for an arbitrary $v\in L^2(\Omega)$. Therefore, we have completed the proof of the lemma.
\end{proof}
\begin{remark}
     The $i$-th eigenvalue of $\mathbf{H}_\varepsilon$ corresponds exactly to the $i$-th eigenvalue of $\chi_\varepsilon \mathbf{\tilde{H}}_\varepsilon \chi_\varepsilon$.
\end{remark}
 By considering Proposition \ref{proposition4.4}, Lemma \ref{lemma4.5}, and Lemma \ref{lemma4.6}, we can deduce the existence of a constant $C$ that is independent of the specific value of $i$ such that
\begin{align}\label{eq4.37}
\lvert \mu_i(\varepsilon)^{-1}- \mu_i^{-1}\rvert \le C (\varepsilon^{1+s}+ \varepsilon + \varepsilon^2 \lvert \log \varepsilon \rvert) \le C \varepsilon
\end{align}
holds.

To establish Theorem \ref{theo2.1}, we require a more precise estimate for the left-hand side of inequality \eqref{eq4.37}. Utilizing the information from \eqref{eq4.37}, we can conclude that when the multiplicity of $\mu_i$ is one, the multiplicity of $\mu_i(\varepsilon)$ is also one for sufficiently small $\varepsilon$. 
\subsection{Perturbation Analysis of \texorpdfstring{$\mathbf{\overline{H}_\varepsilon}$}{Lg}}
In this subsection, we investigate the behaviour of eigenvalues of $\mathbf{\overline{H}}_\varepsilon$ as $\varepsilon$ approaches zero. We introduce the operators $A_0 = \mathbf{K}$ and define the operators $A_1$, $A_2$, and $A_3$ as follows:
\begin{align*}
(A_1f)(x) &= K(x,\tilde{w}) (\mathbf{K}f)(\tilde{w}),\\
(A_2f)(x) &= \langle \nabla_w K(x,w), \nabla_w (\mathbf{K}\xi_\varepsilon f)(\tilde{w}) \rangle  \xi_\varepsilon(x)\\
(A_3f)(x) &= \langle H_w K(x,w), H_w (\mathbf{K}\xi_\varepsilon f) (\tilde{w}) \rangle  \xi_\varepsilon(x).
\end{align*}
We express $\mathbf{\overline{H}_\varepsilon}$ as the sum $\mathbf{\overline{H}_\varepsilon} = A_0 + \bar{g}(\varepsilon)A_1+ h(\varepsilon) A_2 + i(\varepsilon) A_3$,
where 
\begin{eqnarray} \label{eq4.38}
    \bar{g}(\varepsilon)= -\pi \mu_i (M\varepsilon)^{2}.
\end{eqnarray}
Moreover, we introduce the quantities $\lambda(\varepsilon)$ and $\psi(\varepsilon)$, which represent approximate eigenvalues and approximate eigenfunctions of $\mathbf{\overline{H}}_\varepsilon$. They are given by:
\begin{eqnarray*}
\lambda(\varepsilon) &=& \lambda_0 + \bar{g}(\varepsilon) \lambda_1+ h(\varepsilon) \lambda_2 + i(\varepsilon) \lambda_3\\
\psi(\varepsilon) &=& \psi_0 + \bar{g}(\varepsilon) \psi_1+ h(\varepsilon) \psi_2 + i(\varepsilon) \psi_3.
\end{eqnarray*}
Let $\lambda_0$ be a simple eigenvalue of $A_0$. We begin by defining $\psi_0$ as the solution to the equation:
\begin{eqnarray} \label{eq4.39}
(A_0-\lambda_0)\psi_0 = 0, \quad \|\psi_0\|_{L^2(\Omega)} =1.
\end{eqnarray}
Subsequently, we proceed to solve the following equations:
\begin{eqnarray}
(A_0-\lambda_0)\psi_1 =(\lambda_1-A_1)\psi_0, \qquad \langle \psi_0, \psi_1 \rangle_{L^2(\Omega)}=0 \label{eq4.40}\\
(A_0-\lambda_0)\psi_2 =(\lambda_2-A_2)\psi_0, \qquad \langle \psi_0, \psi_2 \rangle_{L^2(\Omega)}=0  \label{eq4.41} \\
(A_0-\lambda_0)\psi_3 =(\lambda_3-A_3)\psi_0, \qquad \langle \psi_0, \psi_3 \rangle_{L^2(\Omega)}=0.  \label{eq4.42}
\end{eqnarray} 
According to the Fredholm alternative theory, the unique solution $\psi_1$, $\psi_2$, $\psi_3$ of equations \eqref{eq4.40}, \eqref{eq4.41}, and \eqref{eq4.42} exists if and only if the following condition holds:
\begin{eqnarray} \label{eq4.43}
\lambda_n= \langle A_n \psi_0, \psi_0 \rangle_{L^2(\Omega)} \qquad (n=1,2,3)
\end{eqnarray}
It is worth noting that $\lambda_0$ is chosen as $\lambda_0= \mu_i^{-1}$, and therefore $\psi_0= \varphi_i$.

\begin{lemma}\label{lemma4.7}
The operator norms of $A_1$, $A_2$, and $A_3$ satisfy the following inequalities:
\begin{enumerate}
    \item For $p > 1$, $\|A_1\|_{L^p(\Omega)} \le C$.
    \item For $p = 2$, $\|A_2\|_{L^p(\Omega)} \le C \lvert \log \varepsilon \rvert$.\\
For $p > 2$, $\|A_2\|_{L^p(\Omega)} \le C\varepsilon^{2/p-1}$.
\item  For $p > 1$, $\|A_3\|_{L^p(\Omega)} \le C\varepsilon^{-2}$.
\end{enumerate}
Here, the constant $C$ is independent of $\varepsilon$.
\end{lemma}
\begin{proof}
	Let $f\in L^p(\Omega)$ be arbitrary. We have:
    \begin{eqnarray*}
    \|A_1\|_{L^p(\Omega)} &\le& C \| G(.,w)\|_{L^p(\Omega)} \|\mathbf{K}f\|_{L^\infty(\Omega)} \\
              &\le& C \|f\|_{L^p(\Omega)} \qquad (p>1).
    \end{eqnarray*}
	Applying Minkowski inequality, we obtain
	\begin{align}\label{eq4.44}
	    &\|A_2f\|_{L^p(\Omega)} =  \left( \int_{\Omega \backslash E_{\eps/2}} \bigg\lvert \sum_{n=1}^{2}   \frac{\partial}{\partial w_n} K(x,\tilde{w})\frac{\partial}{\partial w_n}(\mathbf{K} \xi_\varepsilon f)(\tilde{w}) \bigg\rvert^p dx \right)^{1/p}\nonumber\\
		 &\le   \sum_{n=1}^{2}\left( \int_{\Omega \backslash E_{\varepsilon/2}}  \left\lvert \frac{\partial}{\partial w_n} K(x, \tilde{w}) \right\rvert^p \left\lvert \frac{\partial}{\partial w_n} (\mathbf{K}\xi_\varepsilon f)(\tilde{w})\right\rvert^p dx\right) ^{1/p} \nonumber\\
		&\le   \sum_{n=1}^{2}\left( \int_{\Omega \backslash E_{\varepsilon/2}} \left\lvert \frac{\partial}{\partial w_n} K(x,\tilde{w})\right\rvert^p  dx\right) ^{1/p} \left\lvert \frac{\partial}{\partial w_n} (\mathbf{K}\xi_\varepsilon f)(\tilde{w})\right\rvert.
	\end{align}
	Furthermore, we have
	\begin{equation}\label{eq4.45}
	    \begin{aligned}
	        \left( \int_{\Omega \backslash E_{\varepsilon/2}} \left\lvert \frac{\partial}{\partial w_n} K(x,\tilde{w}) \right\rvert^p  dx\right)^{1/p} &\le  C \left( \int_{\Omega \backslash E_{\varepsilon/2}}  \lvert x-\tilde{w}\rvert^{-p} dx \right) ^{1/p}  \\
			&\le
   \begin{cases}
       C \, \lvert \log \varepsilon \rvert  \qquad\text{if } p=2, \\ 
        C \varepsilon^{2/p -1} \qquad\text{if } p>2 , 
   \end{cases}
	    \end{aligned}
	\end{equation}
	for $n=1,2$. Using H\"older inequality, we have the estimate
	\begin{equation}\label{eq4.46}
	    \begin{aligned}
	      \left\lvert\frac{\partial}{\partial w_n}(\mathbf{K}\xi_\varepsilon f)(\tilde{w}) \right\rvert &\le C \left( \int_{\Omega} \lvert y-\tilde{w}\rvert^{-q}dy \right)^{1/q} \|f\|_{L^p(\Omega)} \\
	&\le  C \|f\|_{L^p(\Omega)} \text{  if } p > 2,  
	    \end{aligned}
	\end{equation} 
  where $q$ satisfies $(1/p) + (1/q)=1$.
  
\noindent By combining equations \eqref{eq4.44}, \eqref{eq4.45}, and \eqref{eq4.46}, we obtain:
	\begin{eqnarray*}
	 \|A_2\|_{L^p(\Omega)}  &\le& \begin{cases}
	     C \lvert \log \varepsilon \rvert  \qquad \text{if }p=2,\\
	 C \varepsilon^{2/p -1} \qquad \text{if } p>2.
	 \end{cases}
	\end{eqnarray*}
Similarly, for $A_3$, we have:
 \begin{eqnarray*}
     	\|A_3 f\|_{L^p(\Omega)} &\le&  \sum_{m,n=1}^{2}\left(\int_{\Omega \backslash E_{\varepsilon/2}}  \left\lvert \frac{\partial^2}{\partial w_m \partial w_n} K(x, \tilde{w}) \right\rvert^p  dx\right) ^{1/p} \left\lvert \frac{\partial^2}{\partial w_m \partial w_n} (\mathbf{K}\xi_\varepsilon f)(\tilde{w})\right\rvert \\
       &\le& C \varepsilon^{-2} \|f\|_{L^p(\Omega)} \qquad \text{if }p>1,
 \end{eqnarray*}
we complete the proof of the lemma.
\end{proof}
 \begin{lemma}\label{lemma4.8}
     For a constant $C$ independent of $\varepsilon$, the following inequalities hold:
\begin{enumerate}
    \item If $p > 1$, then $\|\psi_1\|_{L^p(\Omega)} \le C$.
\item If $p = 2$, then $\|\psi_2\|_{L^p(\Omega)} \le C \lvert\log \varepsilon \rvert$.\\
If $p > 2$, then $\|\psi_2\|_{L^p(\Omega)} \le C \varepsilon^{2/p - 1}$.
\item If $p > 1$, then $\|\psi_3\|_{L^p(\Omega)} \le C \varepsilon^{-2}$.
\end{enumerate}
 \end{lemma}
	
\begin{proof}
We begin by observing the following inequalities:
\begin{eqnarray*}
		\|(\lambda_1 - A_1 )\psi_0\|_{L^2(\Omega)} &\le& C \|A_1\|_{L^2(\Omega)} \le C',\\
		\|(\lambda_2 - A_2 )\psi_0\|_{L^2(\Omega)} &\le& C \|A_2\|_{L^2(\Omega)} \le C \lvert \log \varepsilon\rvert, \\
        \|(\lambda_3 - A_3 )\psi_0\|_{L^2(\Omega)} &\le& C \|A_3\|_{L^2(\Omega)} \le C \varepsilon^{-2}.
	\end{eqnarray*}
	By the Fredholm theory, we obtain the following $L^2(\Omega)$-norm bounds for $\psi_1$ and $\psi_2$:
	\begin{eqnarray*}
	    \|\psi_1\|_{L^2(\Omega)} &\le& C \| \lambda_1- A_1\|_{L^2(\Omega)}  \|\psi_0\|_{L^2(\Omega)} \le C' \\
		\|\psi_2\|_{L^2(\Omega)} &\le& C \| \lambda_2- A_2\|_{L^2(\Omega)}  \|\psi_0 \|_{L^2(\Omega)} \le C \lvert \log \varepsilon\rvert\\
        \|\psi_3\|_{L^2(\Omega)} &\le& C \| \lambda_3- A_3\|_{L^2(\Omega)}  \|\psi_0 \|_{L^2(\Omega)} \le C \varepsilon^{-2}.
	\end{eqnarray*}
	 By similar arguments, we can obtain $L^p$ estimates for $\psi_1$, $\psi_2$, and $\psi_3$.
\end{proof}

\begin{proposition}\label{proposition4.9}
	There exists a constant $C$ independent of $\varepsilon$ such that
	\begin{eqnarray}
		\|(\mathbf{\overline{H}}_\varepsilon- \lambda(\varepsilon)) \psi(\varepsilon)\|_{L^2(\Omega)} \le C \varepsilon^4 \lvert \log \varepsilon\rvert^2\label{eq4.47}
	\end{eqnarray}
	holds.
\end{proposition}
\begin{proof}
    By \eqref{eq4.39}, \eqref{eq4.40}, \eqref{eq4.41} and \eqref{eq4.42}, we have
\begin{align}\label{eq4.48}
    (\mathbf{\overline{H}}_\varepsilon - \lambda(\varepsilon))\psi(\varepsilon) &= \bar{g}(\varepsilon)^2(A_1 - \lambda_1) \psi_1 +h(\varepsilon)^2(A_2-\lambda_2)\psi_2 + i(\varepsilon)^2(A_3-\lambda_3)\psi_2 \nonumber\\
    &+ \bar{g}(\varepsilon)h(\varepsilon)((A_1-\lambda_1)\psi_2+(A_2-\lambda_2)\psi_1) \nonumber\\
     &+ h(\varepsilon)i(\varepsilon)((A_2-\lambda_2)\psi_3+(A_3-\lambda_3)\psi_2) \nonumber\\
      &+ i(\varepsilon)\bar{g}(\varepsilon)((A_3-\lambda_3)\psi_1+(A_1-\lambda_1)\psi_3).
\end{align}
By using Lemma \ref{lemma4.7} and Lemma \ref{lemma4.8}, we can estimate each term in \eqref{eq4.48}. Therefore, we obtain
\begin{equation*}
\|(\mathbf{\overline{H}}\varepsilon - \lambda(\varepsilon))\psi(\varepsilon)\|_{L^2(\Omega)} \le C \varepsilon^4 \lvert \log \varepsilon \rvert^2,
\end{equation*}
which completes the proof of Proposition \ref{proposition4.9}.
\end{proof}

In the following proposition, our goal is to establish an estimate for $\|(\mathbf{K_\varepsilon}-\mathbf{H_\varepsilon})( \chi_\varepsilon\psi(\varepsilon))\|_{L^2(\Omega_{E_\varepsilon})}$.
\begin{proposition} \label{proposition4.10}There exists a constant $C$ independent of $\varepsilon$ such that
\[ \|(\mathbf{K_\varepsilon}-\mathbf{H_\varepsilon})( \chi_\varepsilon\psi(\varepsilon))\|_{L^2(\Omega_{E_\varepsilon})} \le C\varepsilon^3 \lvert \log \varepsilon \rvert\]
  holds.  
\end{proposition}
\begin{proof}
     We fix $\theta \in(0,1)$. Then, by Proposition \ref{proposition4.4}, Lemma \ref{lemma4.8}, Theorem \ref{theo4.1}, and \eqref{eq4.38}, we have
\begin{align*}
    &\|(\mathbf{K_\varepsilon}-\mathbf{H_\varepsilon})( \chi_\varepsilon\psi(\varepsilon))\|_{L^2(\Omega_{E_\varepsilon})} \\
   &\le \|(\mathbf{K_\varepsilon}-\mathbf{H_\varepsilon})( \chi_\varepsilon\varphi_i)\|_{L^2(\Omega_{E_\varepsilon})} \\
   &\qquad + \|(\mathbf{K_\varepsilon}-\mathbf{H_\varepsilon})\|_{L^2(\Omega_{E_\varepsilon})} ( \lvert \bar{g}(\varepsilon)\rvert  \|\psi_1\|_{L^2(\Omega)}+ \lvert h(\varepsilon)\rvert  \|\psi_2\|_{L^2(\Omega)} + \lvert i(\varepsilon)\rvert  \|\psi_3\|_{L^2(\Omega)} )\\
   & \le C( \varepsilon^3 |\log \eps|^2 + \varepsilon^{1+\theta} ( \varepsilon^2 + \varepsilon^2 \lvert \log \varepsilon \rvert + \varepsilon^2)\\
   &\le C \varepsilon^3 |\log \eps|^2.
\end{align*}
The proof of Proposition \ref{proposition4.10} is complete.
\end{proof}
\subsection{Proof of Theorem 4.2}
In this subsection, we aim to compare the operators $\mathbf{H_\varepsilon}$ with $\mathbf{\overline{H}_\varepsilon}$ and establish an estimate for the difference between them. To do so, we introduce the characteristic function $\hat{\chi}_\varepsilon$ of the set $E_\varepsilon$, defined as $\hat{\chi}_\varepsilon = 1- \chi_\varepsilon$.
Then, we define $J_\varepsilon(x,v)= (\chi_\varepsilon \overline{\mathbf{H}}_\varepsilon v - \mathbf{H}_\varepsilon (\chi_\varepsilon v))(x)$ for $v\in L^p(\Omega)$. In Lemma \ref{lemma4.11}, we establish an important estimate for $\|J_\varepsilon(\cdot,v)\|_{L^2(\Omega_{E_\varepsilon})}$.  \begin{lemma}\label{lemma4.11}
    There exists a constant $C$ independent of $\varepsilon$ such that
    \begin{align}\label{eq4.49}
    \|J_\varepsilon(\cdot,v)\|_{L^2(\Omega_{\varepsilon})} \le  C \varepsilon^{2-2/p }  \|v\|_{L^p(\Omega)}
    \end{align}
    holds for any $v \in L^p(\Omega)$ $(p>2)$.
\end{lemma}
\begin{proof}
  To prove this lemma, we first note that
  \begin{align}\label{eq4.50}
     \Delta J_\varepsilon(x,v)=0  \qquad x \in \Omega_{E_\varepsilon} \\ 
     J_\varepsilon(x,v) =0 \qquad x \in \partial \Omega. \nonumber
  \end{align}
	 By applying \eqref{eq4.10}, we have the following expression  
		\begin{align} \label{eq4.51}
		 \frac{\partial}{\partial x_1}  J_\varepsilon(x,v)\vert_{ x=x^*} = \sum_{n=1}^4 J_n(\varepsilon,v),
	\end{align}
	where
 \begin{align*}
     J_1(\varepsilon, v)&= \frac{\partial}{\partial x_1}(\mathbf{K}\hat{\chi}_\varepsilon v)(x^*) - \left( \frac{\partial}{\partial w_1} +  x^*_1 \frac{\partial^2}{{\partial w_1}^2} +x^*_2 \frac{\partial^2}{{\partial w_2 \partial w_1}}
 \right)(\mathbf{K}  \xi_\varepsilon \hat{\chi}_\varepsilon v)(\tilde{w})  \\
	J_2(\varepsilon, v)  &= 2\pi  (\varepsilon M)^2 \frac{\partial}{\partial x_1}\left\langle  \nabla_w S(x^*,\tilde{w}), \nabla_w (\mathbf{K} ( \xi_\varepsilon \hat{\chi}_\varepsilon v))(\tilde{w})\right\rangle  \\
    J_3(\varepsilon, v) &=  (\pi/2)  (\varepsilon M)^4  \frac{\partial}{\partial x_1}\left\langle  H_w S(x^*,\tilde{w}), H_w (\mathbf{K} \xi_\varepsilon \hat{\chi}_\varepsilon v)(\tilde{w})\right\rangle  \nonumber\\
    J_4(\varepsilon, v)   &= -\pi \mu_i (\varepsilon M)^2 \frac{\partial}{\partial x_1} K(x^*,\tilde{w}) (\mathbf{K}v)(\tilde{w}).
 \end{align*}
Note that $E_\varepsilon \subset \Bb(x^*, 2\varepsilon M_\text{max})$ and $(E_\varepsilon \setminus E_{\varepsilon/2}) \subset A(\tilde{w};\frac{\eps}{2} M_{\min},\eps M_{\max})$ for $\tilde{w} = (0,0)$, we can establish the following inequalities:
\begin{equation} \label{eq4.52}
    \begin{aligned}
         \lvert J_1(\varepsilon,v)\rvert &\le C \left( \int_{\Bb(x^*,2\varepsilon M_{\max})} \lvert x^*-y \rvert^{-p'}\,dy\right)^{1/p'} \|v\|_{L^p(\Omega)}\\
    &\quad + C \left( \int_{A(\tilde{w};\frac{\eps}{2} M_{\min},\eps M_{\max}) } \lvert y-w \rvert^{-p'}\,dy\right)^{1/p'} \|v\|_{L^p(\Omega)} \\
    &\quad+ C\varepsilon\left( \int_{A(\tilde{w};\frac{\eps}{2} M_{\min},\eps M_{\max})}\lvert x^*-y \rvert^{-2p'}\,dy\right)^{1/p'} \|v\|_{L^p(\Omega)} \\
    &\le C \varepsilon^{1-2/p}\|v\|_{L^p(\Omega)} \qquad(p>2), 
    \end{aligned}
\end{equation}
\begin{equation}\label{eq4.53}
\begin{aligned}
    \lvert J_2(\varepsilon,v) \rvert  &\le C \lvert h(\varepsilon) \rvert \left( \int_{A(\tilde{w};\frac{\eps}{2} M_{\min},\eps M_{\max})} \lvert y-w \rvert^{-p'}\,dy\right)^{1/p'} \|v\|_{L^p(\Omega)} \\
    &\le C \varepsilon^{3-2/p}  \|v\|_{L^p(\Omega)} \qquad(p>2),
\end{aligned}
\end{equation}\label{eq4.54}
\begin{equation}
    \begin{aligned}
    \lvert J_3(\varepsilon,v) \rvert  &\le C \lvert i(\varepsilon) \rvert \left( \int_{A(\tilde{w};\frac{\eps}{2} M_{\min},\eps M_{\max})} \lvert y-w \rvert^{-2p'}\,dy\right)^{1/p'} \|v\|_{L^p(\Omega)} \\
    &\le C \varepsilon^{4-2/p}  \|v\|_{L^p(\Omega)} \qquad(p>1)
    \end{aligned}
    \end{equation}
    \begin{equation}\label{eq4.55}
        \lvert J_4(\varepsilon,v) \rvert  \le C \varepsilon  \|v\|_{L^p(\Omega)} \qquad(p>1) 
    \end{equation}
where $p'$ satisfies $(1/p) + (1/p')=1$.
\noindent Based on these inequalities, we can summarize the results as follows:
\begin{align} \label{eq4.56}
    \left\lvert 	 \frac{\partial}{\partial x_1}  J_\varepsilon(x,v)\vert_{ x=x^*} \right\rvert \le C \varepsilon^{1-2/p }  \|v\|_{L^p(\Omega)} \qquad(p>2).
\end{align}
Analogously, the upper estimate of $ |\frac{\partial}{\partial x_2}  J_\varepsilon(x,v)\vert$ at $x=x^*$ can be obtained
without any new difficulty 
\begin{align} \label{J_esp_x_2}
    \left\lvert 	 \frac{\partial}{\partial x_2}  J_\varepsilon(x,v)\vert_{ x=x^*} \right\rvert \le C \varepsilon^{1-2/p }  \|v\|_{L^p(\Omega)} \qquad(p>2).
\end{align}
Combining \eqref{eq4.56} and \eqref{J_esp_x_2} yields
\begin{align} \label{J_esp_nu}
    \left\lvert 	 \frac{\partial}{\partial\nu}  J_\varepsilon(x,v)\vert_{ x=x^*} \right\rvert \le C \varepsilon^{1-2/p }  \|v\|_{L^p(\Omega)} \qquad(p>2).
\end{align}
Using \eqref{J_esp_nu} along with Lemma \ref{lemma3.3}, we can conclude that:
\begin{align*}
        \|J_\varepsilon(\cdot,v)\|_{L^2(\Omega_{\varepsilon})} \le  C \varepsilon^{2-2/p } |\log \eps| \|v\|_{L^p(\Omega)} \qquad(p>2),
    \end{align*}
we finish the proof of Lemma \ref{lemma4.11}.
\end{proof}
  
Using the evaluation method similar to the proof of Theorem \ref{theo4.1}, we can derive the following formula
\begin{align} \label{eq4.57}
    \lvert J_1(\varepsilon, \varphi_i)\rvert \le C \varepsilon^2 |\log \eps|.
\end{align}
It is evident that
\begin{align} \label{eq4.58}
     \lvert J_4(\varepsilon, \varphi_i)\rvert \le C \varepsilon^2.
\end{align}
Hence, we have 
\begin{align}\label{eq4.59}
     \lvert J_1(\varepsilon, \varphi_i) + J_4(\varepsilon, \varphi_i)\rvert \le C \varepsilon^2 |\log \eps|.
\end{align}
By summing up \eqref{eq4.50}, \eqref{eq4.52}, \eqref{eq4.53}, \eqref{eq4.57}, \eqref{eq4.58} and \eqref{eq4.59}, we have 
\begin{align*}
     \left\lvert 	 \frac{\partial}{\partial x_1}  J_\varepsilon(x,\varphi_i)\vert_{ x=x^*} \right\rvert \le C \varepsilon^2 |\log \eps|.
\end{align*}
Similarly, we get
\begin{align*}
     \left\lvert 	 \frac{\partial}{\partial x_2}  J_\varepsilon(x,\varphi_i)\vert_{ x=x^*} \right\rvert \le C \varepsilon^2 |\log \eps|.
\end{align*}
Therefore, it follows that 
\begin{align}\label{eq4.60}
     \left\lvert 	 \frac{\partial}{\partial \nu}  J_\varepsilon(x,\varphi_i)\vert_{ x=x^*} \right\rvert \le C \varepsilon^2 |\log \eps|.
\end{align}
Using inequality \eqref{eq4.60} and Lemma \ref{lemma3.3}, we derive 
\begin{align} \label{eq4.61}
        \|J_\varepsilon(\cdot, \varphi_i)\|_{L^2(\Omega_{E_\varepsilon})} \le  C \varepsilon^3 |\log\eps|^2.
    \end{align}
    Consequently, we have established Theorem \ref{theo4.2}.
    
    Furthermore, we aim to estimate $\|J_\varepsilon(\cdot, \psi(\varepsilon))\|_{L^2(\Omega_{E_\varepsilon)}}$. By utilizing equation \eqref{eq4.61}, Lemmas \ref{lemma4.8} and \ref{lemma4.11}, we can derive the following:
    \begin{align*}
       \|J_\varepsilon(\cdot, \psi(\varepsilon))\|_{L^2(\Omega_{E_\varepsilon)}} &\le \| J_\varepsilon(\cdot, \varphi_i)\|_{L^2(\Omega_{E_\varepsilon)}} + \lvert g(\varepsilon) \rvert \|J_\varepsilon(\cdot, \psi_1)\|_{L^2(\Omega_{E_\varepsilon)}} \\
      &\quad + \lvert h(\varepsilon) \rvert \|J_\varepsilon(\cdot, \psi_2)\|_{L^2(\Omega_{E_\varepsilon)}}
       + \lvert i(\varepsilon) \rvert \|J_\varepsilon(\cdot, \psi_3)\|_{L^2(\Omega_{E_\varepsilon)}}\\
       &\le C(\varepsilon^3 | \log \eps|^2 +\varepsilon^{4-2/p
       } +\varepsilon^{6-2/p}) \\
       &\le C \varepsilon^3 |\log\eps|^2.
    \end{align*}
   As a result, we arrive at the subsequent proposition.
    \begin{proposition}\label{proposition4.12}
        There exists a constant $C$ independent of $\varepsilon$ such that the following inequality holds:
        \begin{align*}
            \|(\mathbf{H}_\varepsilon \chi_\varepsilon -\chi_\varepsilon \overline{\mathbf{H}}_\varepsilon) \psi(\varepsilon) \|_{L^2(\Omega_{E_\varepsilon)}} \le C \varepsilon^3 |\log \eps|^2.
        \end{align*}
    \end{proposition}
\section{Proof of Theorem 2.1}
Armed with the estimates established in Section 4, we are now ready to demonstrate the validity of Theorem \ref{theo2.1}. Using Propositions \ref{proposition4.9}, \ref{proposition4.10} and \ref{proposition4.12}, we obtain the following estimate
	\begin{align*} 
		\|(\mathbf{K}_\varepsilon- \lambda(\varepsilon)) (\chi_\varepsilon\psi(\varepsilon))\|_{L^2(\Omega_{E_\varepsilon})} &\le 
		\|(\mathbf{K_\varepsilon}-\mathbf{H_\varepsilon})( \chi_\varepsilon\psi(\varepsilon))\|_{L^2(\Omega_{E_\varepsilon})} \\
		&\quad+\|\mathbf{\overline{H}}_\varepsilon \psi (\varepsilon)-\mathbf{H}_\varepsilon( \chi_\varepsilon\psi(\varepsilon))\|_{L^2(\Omega_{E_\varepsilon})} \\
		&\quad+\|(\mathbf{\overline{H}_\varepsilon}- \lambda(\varepsilon)) \psi(\varepsilon)\|_{L^2(\Omega_{E_\varepsilon})} \nonumber\\
        &\le C \varepsilon^3 |\log \eps|^2.
	\end{align*}	
Note that $\|\psi(\varepsilon)\|_{L^2(\Omega_{E_\varepsilon})} \in (1/2, 2)$ for small $\varepsilon$. As a result, there exists at least one eigenvalue $\lambda^\ast(\varepsilon)$ of $\mathbf{K_\varepsilon}$ that satisfies the condition
\begin{eqnarray}\label{eq5.1}
    \lvert \lambda^\ast(\varepsilon) - \lambda(\varepsilon) \rvert\le C\varepsilon^3 |\log \eps|^2.
\end{eqnarray}
 We can express $\lambda_1$, $\lambda_2$, $\lambda_3$  as follows:
\begin{eqnarray} 
\lambda_1 &=& \lvert (\mathbf{K}\psi_0)(\tilde{w})\rvert^2= \mu_i^{-2}\varphi_i(\tilde{w})^2 \label{eq5.2}\\
\lambda_2 &=& \langle \nabla_w(\mathbf{K} \xi_\varepsilon \psi_0)(\tilde{w}), \nabla_w (\mathbf{K} \xi_\varepsilon\psi_0)(\tilde{w})\rangle \label{eq5.3}\\
&=& \sum_{n=1}^{2} \left( \frac{\partial}{\partial w_n} \int_\Omega K(w,y) \xi_\varepsilon(y) \varphi_i(y) \,dy \right)^2 \vert_{w=\tilde{w}} \nonumber \\
\lambda_3 &=& \langle H_w(\mathbf{K} \xi_\varepsilon \psi_0)(\tilde{w}), H_w (\mathbf{K} \xi_\varepsilon \psi_0)(\tilde{w})\rangle  \nonumber\\
&=& \sum_{m,n=1}^{2} \left( \frac{\partial^2}{\partial w_m \partial w_n} \int_\Omega K(w,y) \xi_\varepsilon(y) \varphi_i(y) \,dy \right)^2 \vert_{w=\tilde{w}}.\label{eq5.4}
\end{eqnarray}
We observe that the second-order partial derivatives can be bounded as follows:
\begin{multline*}
   \left\lvert \frac{\partial^2}{\partial w_m \partial w_n} \int_\Omega K(w,y) \xi_\varepsilon(y) \varphi_i(y) \,dy  \vert_{w=\tilde{w}} \right\rvert\\
   \le C \int_{\Omega \setminus \Bb(w, \frac{\varepsilon}{2} M_\text{min})} \lvert y-\tilde{w}\rvert^{-2} \,dy \le C \lvert\log \varepsilon\rvert \quad(1\le m, n \le 2).
\end{multline*}
Hence, we conclude that
\begin{eqnarray}\label{eq5.5}
    \lambda_3 = \mathcal{O}(\lvert \log\varepsilon\rvert^2).
\end{eqnarray}
Similarly, we obtain the following expression:
\begin{align}\label{eq5.6}
 \frac{\partial}{\partial w_n} \int_\Omega K(w,y) \xi_\varepsilon(y) \varphi_i(y) \,dy  \vert_{w=\tilde{w}} \nonumber\\
= \mu_i^{-1} \frac{\partial}{\partial w_n} \varphi_i(\tilde{w}) + I_1^{(n)}(\varepsilon) +I_2^{(n)}(\varepsilon), 
\end{align}
where
\begin{eqnarray*}
I_1^{(n)}(\varepsilon) &=&  - \frac{\partial}{\partial w_n} \int_\Omega S(w,y)(1-\xi_\varepsilon(y) \varphi_i(y) dy \lvert_{w=\tilde{w}} \\
I_2^{(n)}(\varepsilon) &=&  - \frac{\partial}{\partial w_n} \int_\Omega L(w,y)(1-\xi_\varepsilon(y)) \varphi_i(y) dy \vert_{w=\tilde{w}}
\end{eqnarray*}
for $n=1,2$ and  $ L(w,y) = K(w,y)-S(w,y)= -(2\pi)^{-1}\log \lvert w-y \rvert$.

We can proceed with the estimation of $I_1^{(n)}(\varepsilon)$. Utilizing the estimate $\lvert \frac{\partial S(w,y)}{\partial w_n}\rvert \leq C$ for all $y\in E_\varepsilon$, which follows from the smoothness assumption on $S(w,y)$. Since $\varphi_i(y)$ is an eigenfunction of the Neumann Laplacian, it is a smooth function. Moreover, $E_\varepsilon$ is compact, which implies that $\lvert \varphi_i(y) \rvert \le C$ for all $y\in E_\varepsilon$, and $1-\xi_\varepsilon(y) \leq 1$ for all $y\in E_\varepsilon$. Consequently, we have
\begin{align}\label{eq5.7}
\lvert I_1^{(n)}(\varepsilon)\rvert &\leq \int_{E_\varepsilon} \left \lvert\frac{\partial S(w,y)}{\partial w_n}(1-\xi_\varepsilon(y)) \varphi_i(y)\right \rvert dy  \nonumber\\
&\leq \int_{E_\varepsilon} \left\lvert\frac{\partial S(w,y)}{\partial w_n} \right\rvert \lvert (1-\xi_\varepsilon(y))\rvert \lvert \varphi(y) \rvert dy \nonumber\\
&\leq C \int_{E_\varepsilon} 1 dy \nonumber\\
&\le C \varepsilon^2 \qquad (n=1,2).
\end{align}
Continuing our analysis, we estimate the term 
\begin{equation}\label{eq5.8}
    \begin{aligned} 
\lvert I_2^{(n)}(\varepsilon) \rvert &= \left\lvert - \frac{\partial}{\partial w_n} \int_\Omega L(w,y)(1-\xi_\varepsilon(y)) \varphi_i(y) dy \vert_{w=\tilde{w}} \right\rvert\\
&=  \frac{\partial}{\partial w_n} \int_\Omega \left( (2\pi)^{-1}\log \lvert w-y \rvert\right)(1-\xi_\varepsilon(y) )\varphi_i(y) dy \vert_{w=\tilde{w}}\\
&= \frac{\partial}{\partial w_n} \int_\Omega \left( (2\pi)^{-1}\log \lvert w-y \rvert\right)(1-\xi_\varepsilon(y) )(\varphi_i(y)-\varphi_i(w)) dy \vert_{w=\tilde{w}} \\
&\quad+ \frac{\partial}{\partial w_n} \left(\varphi_i(w)\int_\Omega \left( (2\pi)^{-1}\log \lvert w-y \rvert\right)(1-\xi_\varepsilon(y) ) dy \right)\vert_{w=\tilde{w}} \\
&\le C\varepsilon^2 \lvert \log \varepsilon \rvert \quad (n=1,2).
\end{aligned}
\end{equation}
Combining the estimates \eqref{eq5.3}, \eqref{eq5.6}, \eqref{eq5.7} and \eqref{eq5.8}, we obtain
\begin{eqnarray}\label{eq5.9}
\lambda_2= \mu_i^{-2} \lvert \nabla \varphi_i(\tilde{w})\rvert^2 + \mathcal{O}(\varepsilon^2 \lvert\log \varepsilon \rvert).
\end{eqnarray}
In light of equations \eqref{eq5.2}, \eqref{eq5.5} and \eqref{eq5.9}, we can deduce the following
\begin{align}\label{eq5.10}
    \lambda(\varepsilon) &= \lambda_0 + \bar{g}(\varepsilon) \lambda_1 + h(\varepsilon)\lambda_2 + i(\varepsilon)\lambda_3 \\
&= \lambda_0  -\pi \mu_i(\varepsilon M)^2\lambda_1 + 2\pi (\varepsilon M)^2\lambda_2 + (2/\pi) (\varepsilon M)^4\lambda_3 \nonumber\\
	 &= \mu_i^{-1} - (\pi  \mu_i^{-1}\varphi_i(\tilde{w})^2  - 2 \pi  \mu_i^{-2} \lvert \nabla \varphi_i(\tilde{w})\rvert^2 )  (\varepsilon M)^2 + \mathcal{O}(\varepsilon^4 \lvert\log \varepsilon\rvert^2). \nonumber
\end{align}
Summarizing the estimates with respect to \eqref{eq5.1}, \eqref{eq5.10} and the result \eqref{eq4.37} implies that $\lambda^\ast(\varepsilon)$ must be equal to $\mu_i(\varepsilon)^{-1}$. Therefore, we have
\begin{align*}
\lvert \mu_i(\varepsilon)^{-1} - ( \mu_i^{-1} - (\pi \mu_i^{-1}\varphi_i(\tilde{w})^2  - 2 \pi  \mu_i^{-2} \lvert \nabla \varphi_i(\tilde{w})\rvert^2 )  (\varepsilon M)^2) \rvert 
\le C \varepsilon^3 |\log \eps|^2.
\end{align*}
From this, we deduce the asymptotic expansion for $\mu_i(\varepsilon)$:
\begin{eqnarray*}
	\mu_i(\varepsilon) = \mu_i - ( 2 \pi  \lvert \nabla \varphi_i(\tilde{w})\rvert^2 -\pi \mu_i\varphi_i(\tilde{w})^2 )(\varepsilon M)^2  + O(\varepsilon^3 |\log \eps|^2).
\end{eqnarray*}
Equivalently, \[\mu_i(\varepsilon)  =\mu_i  - (2 \lvert \nabla \varphi_i(\tilde{w})\rvert^2 -\mu_i \varphi_i(\tilde{w})^2) |E|\varepsilon^2  +  \mathcal{O}(\varepsilon^3 |\log \eps|^2)\]
 Hence, we have completed the proof of Theorem \ref{theo2.1}.
\section{Conclusions}
This paper's main contribution is developing a framework to obtain asymptotic formulas for eigenvalues in domains with star-shaped holes, which generalise the spherical hole studied previously. Estimates and proofs are presented in Sections 3-5. The proposed framework can be applied in various scientific and engineering fields. For the future, the author proposes to construct asymptotic formulas for eigenvalues in domains with holes in compact sets with zero Lebesgue measure.
\section*{Acknowledgments}

The author would like to thank the referees for their valuable comments which help to improve the manuscript.

A special thanks goes to Assistant Professor Bao Quoc Tang from the Department for Mathematics and Scientific Computing at the University of Graz for his invaluable insights and comments. 

The paper is completed during the author's visit to the University of Graz, and its hospitality is greatly acknowledged.

This work is supported by grant SGS01/P\v{r}F/2024.

\bibliographystyle{abbrv}
\bibliography{bibliography}

\end{document}